\documentclass[12pt,a4paper,reqno]{amsart}
\usepackage{a4wide,amssymb,bm,mathrsfs,mathtools,mathrsfs,verbatim}
\usepackage[all]{xy}
\mathtoolsset{showonlyrefs}

\newtheorem{thm}{Theorem}[section]
\newtheorem{prop}[thm]{Proposition}
\newtheorem{lemma}[thm]{Lemma}
\newtheorem{defin}[thm]{Definition}
\newtheorem{cor}[thm]{Corollary}
\theoremstyle{definition}
\newtheorem{rema}[thm]{Remark}
\numberwithin{equation}{section}
\newcommand{\triend}{\parbox{2mm}{\hfill} \hfill\mbox{\hspace{0.2mm}}\hfill$\triangle$}
\newenvironment{rem}{\begin{rema}}{\triend\end{rema}}
\newcommand{\coker}{\operatorname{coker}}
\newcommand{\Aut}{\operatorname{Aut}}
\newcommand{\rk}{\operatorname{rk}}
\newcommand{\de}{\mbox{\rm d}}
\newcommand{\Hom}{\operatorname{Hom}}

\newcommand{\Iso}{\operatorname{Iso}}
\newcommand{\GL}{\operatorname{GL}}
\newcommand{\im}{\operatorname{Im}}
\newcommand{\Pic}{\operatorname{Pic}}
\newcommand{\Ol}{\mathcal{O}}

\newcommand{\Q}{\mathcal{Q}}

\newcommand{\A}{\mathcal{A}}

\newcommand{\G}{\mathcal{G}}
\newcommand{\F}{\mathcal{F}}
\newcommand{\I}{\mathcal{I}}
\newcommand{\M}{\mathcal{M}}
\newcommand{\E}{\mathcal{E}}
\newcommand{\U}{\mathcal{U}}
\newcommand{\V}{\mathcal{V}}
\newcommand{\W}{\mathcal{W}}
\newcommand{\K}{\mathcal{K}}
\newcommand{\T}{\mathcal{T}}

\newcommand{\Hg}{\mathcal{H}}
\newcommand{\C}{\mathcal{C}}

\newcommand{\Il}{\mathscr{I}}
\newcommand{\id}{\operatorname{id}}
\newcommand{\lExt}{\operatorname{\mathcal{E}\textit{xt}}}
\newcommand{\lTor}{\operatorname{\mathcal{T}\!\textit{or}}}
\newcommand{\lHom}{\operatorname{\mathcal{H}\!\textit{om}}}

\newcommand{\Spec}{\operatorname{Spec}}

\newcommand{\Mk}{\M^n(r,a,c)}
\newcommand{\Pk}{P_{\vec{k}}}
\newcommand{\Lk}{L_{\vec{k}}}
\newcommand{\Gk}{G_{\vec{k}}}
\newcommand{\Z}{\mathbb{Z}}
\newcommand{\Vek}{\mathbb{V}_{\vec{k}}}

\newcommand{\Com}{\mathbb{C}}
\newcommand{\Pu}{\mathbb{P}^1}
\newcommand{\On}{\Ol_{\Sigma_n}}
\newcommand{\Sc}{\mathfrak{Sch}}
\newcommand{\Ob}{\operatorname{Ob}}
\newcommand{\Fa}{\mathfrak{F}}
\newcommand{\Ua}{\mathfrak{U}}
\newcommand{\Va}{\mathfrak{V}}
\newcommand{\Wa}{\mathfrak{W}}
\newcommand{\ca}{\mathfrak{c}}
\newcommand{\da}{\mathfrak{d}}
\newcommand{\X}{\mathfrak{X}}

\newcommand{\Uk}{\U_{\vec{k}}}
\newcommand{\Vk}{\V_{\vec{k}}}
\newcommand{\Wk}{\W_{\vec{k}}}
\newcommand{\Mek}{M_{\vec{k}}}
\newcommand{\li}{\ell_\infty}
\newcommand{\Ui}{\U_{\vec{k},\infty}}
\newcommand{\Vi}{\V_{\vec{k},\infty}}
\newcommand{\Wi}{\W_{\vec{k},\infty}}

\newcommand{\hN}{\mathfrak{N}_{\vec{k}}}

\newcommand{\Ga}{\mathfrak{G}}
\newcommand{\Hl}{\mathscr{H}}
\newcommand{\bpsi}{\bar{\psi}}
\newcommand{\pu}{\pi}

\newcommand{\tT}{\Sigma_n\times\Pk}
\newcommand{\Tu}{\Xi}
\newcommand{\tTi}{\li\times\Pk}
\newcommand{\Ti}{\Xi_\infty}
\newcommand{\ta}{\tilde{\mathfrak{t}}}
\newcommand{\tu}{\mathfrak{t}}
\newcommand{\ua}{\tilde{\mathfrak{u}}}
\newcommand{\uu}{\mathfrak{u}}
\newcommand{\q}{\bm{q}}
\newcommand{\pr}{\bm{p}}
\newcommand{\tU}{\widetilde{\mathfrak{U}}_{\vec{k}}}
\newcommand{\tV}{\widetilde{\mathfrak{V}}_{\vec{k}}}
\newcommand{\tW}{\widetilde{\mathfrak{W}}_{\vec{k}}}
\newcommand{\tE}{\widetilde{\mathfrak{E}}_{\vec{k}}}
\newcommand{\tUi}{\widetilde{\mathfrak{U}}_{\vec{k},\infty}}
\newcommand{\tVi}{\widetilde{\mathfrak{V}}_{\vec{k},\infty}}
\newcommand{\tWi}{\widetilde{\mathfrak{W}}_{\vec{k},\infty}}
\newcommand{\tEi}{\widetilde{\mathfrak{E}}_{\vec{k},\infty}}
\newcommand{\tMo}{\widetilde{\mathbb{M}}_{\vec{k}}}
\newcommand{\tMoi}{\widetilde{\mathbb{M}}_{\vec{k},\infty}}

\newcommand{\dT}{\Upsilon}
\newcommand{\dTi}{\Upsilon_\infty}
\newcommand{\dt}{\breve{\mathfrak{t}}}
\newcommand{\du}{\breve{\mathfrak{u}}}

\newcommand{\Mo}{\mathbb{M}_{\vec{k}}}
\newcommand{\Moi}{\mathbb{M}_{\vec{k},\infty}}
\newcommand{\dE}{\breve{\mathfrak{E}}_{\vec{k}}}
\newcommand{\dEi}{\breve{\mathfrak{E}}_{\vec{k},\infty}}

\newcommand{\dMo}{\breve{\mathbb{M}}_{\vec{k}}}
\newcommand{\dMoi}{\breve{\mathbb{M}}_{\vec{k},\infty}}

\newcommand{\tTheta}{\widetilde{\Theta}_{\vec{k}}}

\newcommand{\Sym}{\operatorname{Sym}}

\newcommand{\Thu}{\Theta_{\vec{k}}}
\newcommand{\Eu}{\ensuremath{\mathfrak{E}_{\vec{k}}}}

\newcommand{\trho}{\tilde{\rho}}
\newcommand{\rhoi}{\rho_{\infty}}

\newcommand{\tA}{\widetilde{A}_{\vec{k}}}
\newcommand{\tB}{\widetilde{B}_{\vec{k}}}

\newcommand{\Uu}{\mathfrak{U}_{\vec{k}}}
\newcommand{\Vu}{\mathfrak{V}_{\vec{k}}}
\newcommand{\Wu}{\mathfrak{W}_{\vec{k}}}
\newcommand{\Eui}{\mathfrak{E}_{\vec{k},\infty}}
\newcommand{\Uui}{\mathfrak{U}_{\vec{k},\infty}}
\newcommand{\Vui}{\mathfrak{V}_{\vec{k},\infty}}
\newcommand{\Wui}{\mathfrak{W}_{\vec{k},\infty}}

\newcommand{\Al}{\mathscr{A}}

\newcommand{\Au}{A_{\vec{k}}}
\newcommand{\Bu}{B_{\vec{k}}}
\newcommand{\Aui}{A_{\vec{k},\infty}}
\newcommand{\Bui}{B_{\vec{k},\infty}}

\newcommand{\Ext}{\operatorname{Ext}}
\newcommand{\Y}{\mathfrak{Y}}
\newcommand{\Wek}{\mathbb{W}_{\vec{k}}}
\newcommand{\Sl}{\mathscr{S}}
\newcommand{\fb}{\bar{f}}
\newcommand{\ev}{\operatorname{ev}}
\newcommand{\Ll}{\mathscr{L}}

\newcommand{\tf}{\tilde{f}}

\newcommand{\Kl}{\mathscr{K}}
\newcommand{\Ma}{\bm{\mu}}
\newcommand{\oMa}{\bm{\nu}}
\newcommand{\Mod}{{\M}\mspace{-1mu}od_{\vec{k}}}
\newcommand{\Sets}{\mathfrak{Sets}}

\newcommand{\lhra}{\lhook\joinrel\relbar\joinrel\rightarrow}

\begin{document}
\begin{flushright}
This is a revision of the paper published in \\ Adv. Geom., {\bf 15} (2015), pp.~55--76
\end{flushright}
\title[Framed sheaves on Hirzebruch surfaces]{\large MONADS FOR FRAMED SHEAVES\\[10pt]  ON HIRZEBRUCH SURFACES}
\bigskip
\date{Revised \today}
\thanks{E-mail: {\tt bartocci@dima.unige.it, bruzzo@sissa.it, clsrava@gmail.com}\\ \indent
Research partially supported by {\sc prin}   ``Geometria delle variet\`a  algebriche
e dei loro spazi dei moduli'' and by {\sc gnsaga-in}d{\sc am}.  
 U.B. is a member of the {\sc vbac} group.}

\maketitle
\markright{\sc C. Bartocci, U. Bruzzo, C.L.S. Rava}
\thispagestyle{empty}

\begin{center}
{\sc Claudio Bartocci,$^\P$ Ugo Bruzzo,$^{\S\ddag\star}$ and Claudio L. S. Rava$^\P$}  \\[10pt]  \small 
$^\P$Dipartimento di Matematica, Universit\`a di Genova, \\Via Dodecaneso 35, 16146 Genova, 
Italia\\[3pt]
$^\star$Departamento de Matem\'atica, Universidade Federal de Santa Catarina, \\ 
  Campus Universit\'ario Trindade,
  88040-900 Florian\'opolis SC, Brasil \\[3pt]
  $^{\S}$SISSA (Scuola Internazionale Superiore di Studi Avanzati),\\   Via Bonomea 265, 34136 Trieste, Italia\\[3pt]
  $^\ddag$Istituto Nazionale di Fisica Nucleare, Sezione di Trieste
\end{center}

\bigskip\bigskip

\begin{abstract}  
We define  monads for framed torsion-free sheaves on Hirzebruch surfaces and use them to construct moduli spaces for these objects.  These moduli spaces are  smooth algebraic varieties, and we show that they are fine  by constructing a universal monad.
\end{abstract}

{\footnotesize
\setcounter{tocdepth}{1}
\tableofcontents
}

\section{Introduction}
Moduli spaces of framed sheaves over projective surfaces have been the object of some interest over the last few years. In the case of the complex projective plane --- generalizing the classical result in \cite{Do} ---  these moduli spaces are resolutions of singularities of the moduli space of ideal instantons on the four-sphere $S^4$ \cite{Nakabook},
and as such, they have been used to compute Nekrasov's partition function, i.e.,   the partition function of a (suitably twisted) $N=2$ topological super Yang-Mills theory  (see \cite{Ne,BFMT} and also \cite{GaLiu} for the more general case of toric surfaces).

More generally, they are at the basis of the so-called instanton counting \cite{NY}. Fine moduli spaces of framed sheaves were constructed by Huybrechts and Lehn \cite{HL1,HL2} by introducing a stability condition.  Bruzzo and Markushevich \cite{BruMar} showed that framed sheaves  on projective surfaces, even without a stability condition, give  rise to fine moduli spaces.

Moduli spaces of framed bundles on the projective plane {were} considered by Donaldson  \cite{Do,Do-Kro}, { who established their
isomorphism with the moduli of framed instantons on $S^4$}. This was extended by King to the case of the blowup of the projective plane at a point  \cite{Ki} and by Buchdahl \cite{Bu2} to the case of multiple blowups. The degenerate case (including torsion-free sheaves) was first considered, as already cited, by Nakajima \cite{Nakabook} for the projective plane, and then by Nakajima and Yoshioka \cite{NY} for the blown-up plane. The case of multiple blowups was studied by Henni \cite{Henni}.

In this paper we consider torsion-free sheaves on a Hirzebruch surface $\Sigma_n$, for $n> 0$, that are framed to the trivial bundle on a generic  line  
$\ell_\infty$ in $\Sigma_n$. We    construct  a moduli space  for  such sheaves  by using monads.  This   allows us to obtain a   moduli space which is a smooth quasi-projective variety. These moduli spaces are  also   shown to be fine, and this   implies  that they are isomorphic to the moduli spaces constructed in \cite{BruMar} (and   therefore embed as   open subschemes  into Huybrechts-Lehn's moduli spaces).

The monads we use generalize to   torsion-free sheaves  those introduced by Buchdahl \cite{Bu} for the locally-free case (he was actually interested in  $\mu$-stable vector bundles on Hirzebruch surfaces with $c_1=0$, $c_2=2$). Indeed, we show in Corollary \ref{pro2} that for any framed torsion-free sheaf $(\E,\theta)$ on $\Sigma_n$, the underlying sheaf $\E$ is isomorphic to the cohomology of a monad
\begin{equation}
\xymatrix{
\Uk \ar[r]^-\alpha & \Vk \ar[r]^-\beta & \Wk
}
\label{eqI1}
\end{equation}
whose terms  depend only on the Chern character $\textrm{ch}(\E)=(r,aE,-c-\frac12na^2)$ (here $E$ is the exceptional curve in $\Sigma_n$, i.e., the unique irreducible   curve in $\Sigma_n$ squaring to $-n$, and we   put $\vec{k}=(n,r,a,c)$). 

We denote by $\Lk$ the subset of $\Hom(\Uk,\Vk)\oplus\Hom(\Vk,\Wk)$ formed by pairs $(\alpha,\beta)$ coming
from a framed sheaf $(\E,\theta)$. We  prove in Proposition \ref{6prop3} and   Lemma \ref{6lm10} that $\Lk$ is a smooth variety. We construct a principal $\GL(r,\Com)$-bundle $\Pk$ over $\Lk$ whose fibre over a point $(\alpha,\beta)$ can be identified with the space of framings for the cohomology of the complex \eqref{eqI1}. 

The group $\Gk$ of isomorphisms of monads of the form \eqref{eqI1} acts on $\Pk$, and  the moduli space $\Mk$ of framed sheaves on $\Sigma_n$ with  given Chern character is set-theoretically defined as the quotient $\Pk/\Gk$. Theorem \ref{proMk} proves that $\Mk$ inherits from $\Pk$ a structure of smooth algebraic variety.  As a corollary,  $\Mk$ turns out to be irreducible. Moreover, Lemma \ref{pro3} states that two monads of the form \eqref{eqI1} are isomorphic if and only it their cohomologies are isomorphic, and this ensures that there is a bijection between $\Mk$ and set of isomorphism classes 
of framed sheaves on $\Sigma_n$ (isomorphisms of framed sheaves are introduced in Definition \ref{DEf3}). This enables us to show that the set of these classes is nonempty if and only if  $ 2c \ge na(1-a)$ (hence, at least when $a\ne 0$, it turns out that the moduli space can be empty even if it has positive expected dimension).

We prove   the fineness of the moduli space by constructing a universal family $\left(\Eu,\Thu\right)$ of framed sheaves on $\Sigma_n$ parametrized by $\Mk$ (for the precise notion of family see Definition \ref{defFam}). 

These results are the basis for further work where we give a detailed description of the moduli spaces when the topological invariants satisfy the lower bound $ 2c = na(1-a)$. Moreover, considering   the rank one case, they allow us to   obtain a rather  explicit construction of the Hilbert schemes of points of the total spaces of the line bundles $\Ol_{\mathbb P^1}(-n)$. In both cases the results are achieved 
by using explicit ADHM descriptions \cite{new}.

By  ``scheme'' we   mean a noetherian reduced scheme of finite type over $\Com$. If  $X$ and $S$ are schemes,     $\F$ is a sheaf of $\Ol_{X\times S}$-modules, and $F$ is a morphism between two such sheaves, we shall denote by  $\F_s$ (resp.~$F_s$) the restriction of $\F$ (resp.~$F$) to the fibre of $X\times S\longrightarrow S$ over the point $s\in S$.

{\bf Acknowledgments.} We thank A.~A.~Henni, B.~Kenda, V.~Lanza, D.~Markushevich and J.~Scalise for useful discussions. We also thank the referee for the careful reading of our manuscript which helped us to considerably improve the presentation. 

\medskip
\section{Monads}
 A monad $M$ on a scheme $T$ is a three-term complex of locally-free $\mathcal O_T$-modules, having nontrivial cohomology only in the middle term:
\begin{equation}
M\,:\qquad\xymatrix{
0 \ar[r]& \mathcal{U}\ar[r]^a &
\mathcal{V}
\ar[r]^b &\mathcal{W} \ar[r] &0\,.
}
\label{eq384}
\end{equation}
The cohomology of the monad will be denoted by 
 $\mathcal{E}(M)$. It is a coherent $\mathcal O_T$-module.
 A \emph{morphism} (\emph{isomorphism}) \emph{of monads} is     a morphism (isomorphism) of complexes.
 
The \emph{display} of the monad \eqref{eq384} is 
the commutative diagram (with exact rows and columns)
\begin{equation}
\begin{gathered}
\xymatrix{
  & & 0 \ar[d] & 0 \ar[d] &  \\
0 \ar[r] & \mathcal{U} \ar@{=}[d] \ar[r]^a & \mathcal{B} \ar[d] \ar[r] & \mathcal{E} \ar[d] \ar[r] & 0 \\
0 \ar[r] & \mathcal{U} \ar[r]^a & \mathcal{V} \ar[d]^b \ar[r] & \mathcal{A} \ar[d]^{\nu} \ar[r] & 0 \\  \
  & & \mathcal{W} \ar@{=}[r] \ar[d] & \mathcal{W} \ar[d] & \\
  & & 0 & 0 & 
}
\end{gathered}
\label{eq44}
\end{equation}
where $\mathcal{A}:=\coker a $, $\mathcal{B}:=\ker b $, $\mathcal{E}=\mathcal{E}(M) $, and all morphisms are naturally induced.
 
 Let $T=X\times S$, where $X$ is a smooth connected projective variety over $\Com$, and $S$ a scheme. We denote by $t_i$, $i=1,2$ the canonical projections onto the first and second factor, respectively. If  $\F$ is an $\Ol_{T}$-module,
we denote by $\F_s$ its restriction to the fibre of $T$ over $s\in S$.

\begin{lemma} \label{lmInjSurj}
 Let $\F$ and $\G$ be coherent $\Ol_{T}$-modules,  flat on $S$, and let
 $\varphi\colon\F\to \G$ be a morphism. If for every closed point  $s\in S$ the restricted morphism $\varphi_{s}$ is injective (surjective), then $\varphi$ is injective (surjective).
\end{lemma}
\begin{proof}
Denote $\K=\ker\varphi$ and $\Q=\coker\varphi$. Let $\Ol_{X}(1)$ be an ample line bundle on $X$, so that   the line bundle $\Ol_{T}(1)=t_{1}^{*}\Ol_{X}(1)$ is relatively ample. By Serre's theorem, for $m\gg 0$ one has an exact sequence of coherent $\Ol_{S}$-modules
\begin{equation}
\xymatrix{
0 \ar[r] & t_{2*}(\K(m)) \ar[r] & t_{2*}(\F(m)) \ar[rr]^-{t_{2*}(\varphi(m))} & & t_{2*}(\G(m)) \ar[r] & t_{2*}(\Q(m)) \ar[r] & 0\,,
}
\label{eqLSeq}
\end{equation}
where   $\varphi(m)=\varphi\otimes_{\Ol_{T}}\id_{\Ol_{T}(m)}$. Moreover there are surjections
\begin{equation}
t_{2}^{*}[t_{2*}(\K(m))] \twoheadrightarrow \K(m)\,,
\qquad 
t_{2}^{*}[t_{2*}(\Q(m))] \twoheadrightarrow \Q(m)\,.
\label{eqSurj}
\end{equation}
We may assume that the sheaves $t_{2*}(\F(m))$ and $t_{2*}(\G(m))$ are locally free, and moreover
that \begin{equation}
H^{i}(T_{s},\F_{s}(m))=H^{i}(T_{s},\G_{s}(m))=H^{i}(T_{s},\K_{s}(m))=0
\label{eqHiFG}
\end{equation}
for $i>0$, and for all $s\in S$.

For all closed points $s\in S$ one has a commutative diagram
\begin{equation}
\xymatrix{
t_{2*}(\F(m))\otimes k(s) \ar[rrr]^-{t_{2*}(\varphi(m))\otimes k(s)} \ar[d] & & & t_{2*}(\G(m))\otimes k(s) \ar[d]\\
H^{0}(T_{s},\F_{s}(m)) \ar[rrr]^{H^{0}(T_{s},\varphi_{s}(m))} & & & H^{0}(T_{s},\G_{s}(m))
}
\label{eqDia}
\end{equation}
where the vertical arrows are the natural ones. By eq.~\eqref{eqHiFG}  they  are isomorphisms.
 
If $\varphi_{s}$ is injective,  eq.~\eqref{eqDia} implies that $t_{2*}(\varphi(m))\otimes k(s)$ is injective. It is easy to see that this implies
\begin{equation}
t_{2*}(\K(m))\otimes k(s)\simeq\lTor^{\,\Ol_{S}}_{1}(\im[t_{2*}(\varphi(m))],k(s))\,.
\end{equation}
There is an open subset $U\subseteq S$  where the right-hand side of this equation is zero. By Nakayama's lemma, the stalks of $t_{2*}(\K(m))$ at the closed points
of $U$ vanish. By \cite[Lemma~2.8]{Eis} it follows that $t_{2*}(\K(m))|_{U}=0$. Since $t_{2*}(\K(m))$ is a subsheaf of a locally free sheaf it   vanishes. The injectivity of $\varphi$ follows from \eqref{eqSurj}.

If $\varphi_{s}$ is surjective, eq.~\eqref{eqHiFG} implies that the morphism $H^{0}(T_{s},\varphi_{s}(m))$ is surjective, and from eq.~\eqref{eqDia} one deduces that   the morphism $t_{2*}(\varphi(m))\otimes k(s)$ is surjective as well. By Nakayama's lemma, the germ $t_{2*}(\varphi(m))_{s}$ is surjective, and \cite[Cor.~2.9]{Eis} implies that $t_{2*}(\varphi(m))$ is surjective, whence
$t_{2*}(\Q(m))=0$.
The thesis follows from eq.~\eqref{eqSurj}.
\end{proof}

Let
\begin{align*}
M&\colon\qquad
\xymatrix{
0 \ar[r] & \U \ar[r]^-{a} & \V \ar[r]^-{b} & \W \ar[r] & 0
}
\end{align*}
be a monad on $T$. We call $\E$  its cohomology, which we assume to be flat on $S$.
\begin{lemma}
\label{lmMs} \label{lmFlat}
For all points  $s\in S$  the restricted complex $M_{s}$  is a monad, whose cohomology is isomorphic to  $\E_{s}$.
\end{lemma}
\begin{proof}
Let us consider the short exact sequence
\begin{equation}\label{lemma2.2}
\xymatrix{
0 \ar[r] & \U \ar[r]^-{a} & \ker b \ar[r] & \E \ar[r] & 0\,.
}
\end{equation}
If $s\in S$, the restricted morphism $a_{s}$ is injective by Lemma 2.1.4 in \cite{HL-book} as $\ker b$ and $\E$ are flat on $S$. The thesis follows easily.
\end{proof}

\begin{prop}  \label{teo1'}
Let $M$, $M'$ be two monads on $T$, and $\E$, $\E'$ their cohomologies, which we assume to be flat on $S$. 
The sheaf $\lHom_{\Ol_{T}}(\E,\E')$ is flat on $S$. Moreover, for every $s\in S$ there is an isomorphism
$\lHom_{\Ol_{T}}(\E,\E')_{s}\simeq \lHom_{\Ol_{T_{s}}}(\E_{s},\E'_{s})$.
\label{eqPropFl}
\end{prop}
\begin{proof}
We consider the monads $M$, $M'$ as complexes
of $\Ol_T$-modules with nonzero entries in degrees -1, 0 and 1. Let
$\T^\bullet$ be the total complex of the double complex $\lHom_{\Ol_{T}}(M,M')$,
i.e.,
$$\T^i = \bigoplus_{p+q=i} \lHom_{\Ol_{T}}(M^{-q},M'^p)$$
with differentials $D^i\colon \T^i \to \T^{i+1}$ defined as usual. Since the terms of $M$ and the sheaves $\im a$ and $\im b$ are  locally free, we can apply the ``dual K\"unneth Theorem'' (see 
e.g.~\cite[Exercise 3.6.1]{Weib})
to $(\T^{\bullet},D^{\bullet})$, obtaining 
\begin{equation*}
\Hg^{i}(\T^{\bullet},D^{\bullet})=
\begin{cases}
\lExt^{1}_{\Ol_{T}}(\E,\E') &\qquad\text{if $i=1$}\\
\lHom_{\Ol_{T}}(\E,\E') &\qquad\text{if $i=0$}\\
0 &\qquad\text{otherwise.}
\end{cases}
\end{equation*}
One deduces the short exact sequences
\begin{gather}
\xymatrix{
0 \ar[r] & \ker D^{1} \ar[r] & \T^{1} \ar[r]^-{D^{1}} & \T^{2} \ar[r] & 0
}
\label{eqSeq1}
\\
\xymatrix{
0 \ar[r] & \im D^{0} \ar[r]^-{\kappa} & \ker D^{1} \ar[r] & \lExt^{1}_{\Ol_{T}}(\E,\E')  \ar[r] & 0
}
\label{eqSeqGlE}
\\
\xymatrix{
0 \ar[r] & \ker D^{0} \ar[r] & \T^{0} \ar[r]^-{D^{0}} & \im D^{0} \ar[r] & 0
}
\label{eqSeq0}
\\
\xymatrix{
0 \ar[r] & \im D^{-1} \ar[r]^-{\iota} & \ker D^{0} \ar[r] & \lHom_{\Ol_{T}}(\E,\E') \ar[r] & 0\,.
}
\label{eqSeqGl}
\end{gather}
Since the terms of  $\T^{\bullet}$ are locally free, $\ker D^{1}$ is locally free as well.

 Notice that for all $s\in S$ 
\begin{equation}
(\im D^{0})_{s}= \im D^{0}_{s}\,,\qquad (\ker D^{1})_{s}=\ker D^{1}_{s}\,.
\label{eqRestrE}
\end{equation}
Here the first isomorphism is a consequence of the right exactness of the functor $-\otimes_{\Ol_{T}}\Ol_{T_{s}}$, while the second    comes from eq.~\eqref{eqSeq1}.  By restricting \eqref{eqSeqGlE} to $T_{s}$ and   using \eqref{eqRestrE} one gets
\begin{equation}
\xymatrix{
\im D^{0}_{s} \ar[r]^-{\kappa_{s}} & \ker D^{1}_{s} \ar[r] & \lExt^{1}_{\Ol_{T}}(\E,\E')_{s} \ar[r] & 0
}\,.
\label{eqSeqGlEs}
\end{equation}
By Lemma \ref{lmMs} one can repeat the construction for the restricted sheaves $\E_{s}$ and $\E'_{s}$ on $T_{s}$, starting from the double complex $\lHom_{\Ol_{T_{s}}}(M_{s},M'_{s})$. One gets the short exact sequence
\begin{equation}
\xymatrix{
0 \ar[r] & \im D^{0}_{s} \ar[r]^-{l_{s}} & \ker D^{1}_{s} \ar[r] & \lExt^{1}_{\Ol_{T_{s}}}(\E_{s},\E'_{s}) \ar[r] & 0
}\,.
\label{eqSeqLocE}
\end{equation}
By comparing \eqref{eqSeqGlEs} and \eqref{eqSeqLocE} it is easy to deduce that $\kappa_{s}=l_{s}$, which implies that $\kappa_{s}$ is injective for all points $s\in S$. From \cite[Lemma~2.1.4]{HL-book} it follows that $\lExt^{1}_{\Ol_{T}}(\E,\E')$ is flat on $S$, and by comparing the two previous exact sequences one also gets  the isomorphism
\begin{equation*}
\lExt^{1}_{\Ol_{T}}(\E,\E')_{s}\simeq \lExt^{1}_{\Ol_{T_{s}}}(\E_{s},\E'_{s})\,.
\end{equation*}
Since $\ker D^{1}$ and $\T^{0}$ are locally free, by applying \cite[Prop.~III.9.1A.(e)]{Har} to eqs. \eqref{eqSeqGlE} and \eqref{eqSeq0} one deduces that $\im D^{0}$ and $\ker D^{0}$ are flat on $S$.

Analogously to \eqref{eqRestrE}, for all $s\in S$ we have
\begin{equation}
(\im D^{-1})_{s}= \im D^{-1}_{s}\,,\qquad (\ker D^{0})_{s}=\ker D^{0}_{s}\,.
\label{eqRestr}
\end{equation}
Again, the first isomorphism is a consequence of the right exactness of the functor $-\otimes_{\Ol_{T}}\Ol_{T_{s}}$. To prove the second isomorphism, we claim that
\begin{equation*}
\lTor^{\Ol_{T}}_{i}(\im D^{0},\Ol_{T_{s}})=0
\end{equation*}
for all $i>0$. Indeed the definition of $T_{s}$ implies the isomorphism $-\otimes_{\Ol_{T}}\Ol_{T_{s}}\simeq -\otimes_{\Ol_{S}}k(s)$. The claim is a consequence of the $S$-flatness of $\im D^{0}$, and the second isomorphism in \eqref{eqRestr} follows from eq. \eqref{eqSeq0}. By restricting \eqref{eqSeqGl} to $T_{s}$ and   using \eqref{eqRestr} one gets
\begin{equation}
\xymatrix{
\im D^{-1}_{s} \ar[r]^-{\iota_{s}} & \ker D^{0}_{s} \ar[r] & \lHom_{\Ol_{T}}(\E,\E')_{s} \ar[r] & 0
}\,.
\label{eqSeqGls}
\end{equation}
Again one can repeat the construction for the restricted sheaves $\E_{s}$ and $\E'_{s}$ on $T_{s}$, getting the short exact sequence
\begin{equation}
\xymatrix{
0 \ar[r] & \im D^{-1}_{s} \ar[r]^-{j_{s}} & \ker D^{0}_{s} \ar[r] & \lHom_{\Ol_{T_{s}}}(\E_{s},\E'_{s}) \ar[r] & 0
}\,.
\label{eqSeqLoc}
\end{equation}
By comparing \eqref{eqSeqGls} and \eqref{eqSeqLoc} it is easy to deduce that $\iota_{s}=j_{s}$, which implies that $\iota_{s}$ is injective for all points $s\in S$. From \cite[Lemma~2.1.4]{HL-book} it follows that $\lHom_{\Ol_{T}}(\E,\E')$ is flat on $S$, and by comparing the two previous exact sequences one also gets  the isomorphism
\begin{equation*}
\lHom_{\Ol_{T}}(\E,\E')_{s}\simeq \lHom_{\Ol_{T_{s}}}(\E_{s},\E'_{s})\,.
\end{equation*}
\end{proof}

\medskip
\section{Statement of the main result}\label{Main result}
Let $\Sigma_n$ be the $n$-th Hirzebruch surface, i.e., the projective closure of the total space of the line bundle $\mathcal{O}_{\mathbb{P}^1}(-n)$. We denote by $F$ the class in $\Pic(\Sigma_n)$  of the  natural ruling $\Sigma_n\longrightarrow\Pu$, by $H$ the class of the section of the ruling squaring to $n$, and by $E$    the class of the section squaring to $-n$. One has $E=H-nF$ and $F^2=0$, $H\cdot F=1$.
We fix a curve  $\ell_{\infty}\simeq\Pu$ in $\Sigma_n$ linearly equivalent to $H$ and call it   the ``line at infinity''.

One has $\Pic(\Sigma_n)=\Z H\oplus\Z F$. For any  sheaf $\E$  of $\On$-modules we shall write
\begin{equation*}
\mathcal{E}(p,q):=\mathcal{E}\otimes\On(pH+qF)\qquad p,q\in \mathbb{Z}\,.
\end{equation*}

\begin{lemma} \label{lmVanH}
\begin{align*}
H^0(\On(p,q))\neq0\qquad&\text{if and only if}\qquad
\begin{cases}
p\geq0\\
np+q\geq0\,;
\end{cases}\\
H^1(\On(p,q))\neq0\qquad&\text{if and only if}\qquad
\begin{cases}
p\geq0\\
q\leq-2
\end{cases}
\qquad\text{or}\qquad
\begin{cases}
p\leq-2\\
q\geq n\,;
\end{cases}\\
H^2(\On(p,q))\neq0\qquad&\text{if and only if}\qquad
\begin{cases}
p\leq-2\\
np+q\leq-(n+2)\,.
\end{cases}
\end{align*}
\end{lemma}
\begin{proof}  Similar to King's proof for the case $n=1$ \cite[pp.~22-23]{Ki}. \end{proof}
Among the several definitions of framed sheaves available in the literature, we shall adopt the following.
 \begin{defin}
\label{DEf2}
A \emph{framed sheaf} is a pair $(\mathcal{E},\theta) $, where
\begin{enumerate}
 \item 
$\mathcal{E}$ is a torsion-free sheaf on $\Sigma_n$ such that
\begin{equation}
\mathcal{E}|_{\ell_{\infty}}\simeq\Ol_{\li}^{\oplus r},
\label{6eq34}
\end{equation}
with $r=\rk(\E)$.
\item
$\theta $ is a fixed isomorphism $\theta:\mathcal{E}|_{\ell_{\infty}}\stackrel{\sim}{\longrightarrow}\Ol_{\li}^{\oplus r}$.
\end{enumerate}
\end{defin}
Condition \eqref{6eq34} implies $c_1({\mathcal{E}})\propto E $.
The isomorphism $\theta$ is the so-called \emph{framing at infinity}.
By ``sheaf trivial at infinity'' we shall mean a sheaf satisfying condition \eqref{6eq34} (without any assigned framing).

\begin{defin}
\label{DEf3}
An isomorphism $\Lambda$ between two framed sheaves $(\mathcal{E},\theta) $ and $(\mathcal{E}',\theta') $ is an isomorphism $\Lambda: \mathcal{E}\stackrel{\sim}{\longrightarrow}\mathcal{E}'$ such that the following diagram commutes:
\begin{equation*}
\xymatrix{
{\mathcal{E}}|_{\ell_{\infty}}\ar[r]^{\theta}\ar[d]_{\Lambda_\infty} & \Ol_{\li}^{\oplus r} \\
{\mathcal{E}}'|_{\ell_{\infty}}\ar[ru]_{\theta'} &  }
\end{equation*}
where $\Lambda_\infty:=\Lambda|_{\ell_{\infty}} $.
\end{defin} 
Let $\Mk$ be  the set of isomorphism classes of framed sheaves on $\Sigma_n$ having rank $r$, first Chern class $aE$, and second Chern class $c$: we shall prove that this set can be endowed with a structure of a smooth algebraic variety. We restrict ourselves to the case $n\geq1$ and assume that the framed sheaves are normalized in such a way that $0\leq a\leq r-1 $.

In order to simplify the statement of Theorem \ref{teo2} we introduce some notation.
\begin{itemize}\label{pagcond}
 \item 
We denote by $\vec{k}$ a quadruple $(n,r,a,c)$, and define $k_i\in\Z$, $i=1,\ldots,4$ as follows:
\begin{equation}
k_1=c+\dfrac{1}{2}na(a-1)\qquad\text{and}\qquad
\begin{cases}
\begin{aligned}
k_2&=k_1+na\\
k_3&=k_1+(n-1)a\\
k_4&=k_1+r-a\,.
\end{aligned}
\end{cases}
\label{eqki}
\end{equation}
The other definitions needed to state Theorem \ref{teo2} make sense only under the assumption $k_1\geq0$. The Theorem itself will give a deeper meaning to this inequality.
\item
We introduce the locally-free sheaves
\begin{equation}
\left\{
\begin{aligned}
\Uk&:=\On(0,-1)^{\oplus k_1}\\
\Vk&:=\On(1,-1)^{\oplus k_2} \oplus \On^{\oplus k_4}\\
\Wk&:=\On(1,0)^{\oplus k_3}.
\end{aligned}
\right.
\label{eqUVWk}
\end{equation}
We shall write $\Ui$ in place of $\left.\Uk\right|_{\li}$, etc.
\item We introduce   the vector space
\begin{equation*}
\Vek:=\Hom\left(\Uk,\Vk\right)\oplus\Hom\left(\Vk,\Wk\right)\,,
\end{equation*}
whose elements will be denoted by $(\alpha,\beta)$.
\item
Let $\overline{L}_{\vec{k}}$ be the affine subvariety of  $\Vek$ cut by the equation $\beta\circ\alpha=0$. One has
 the associated complex
\begin{equation*}
\xymatrix{
M(\alpha,\beta): & \Uk \ar[r]^-\alpha & \Vk \ar[r]^-\beta & \Wk\,.
}
\end{equation*}
\item We define the quasi-affine variety $\Lk$ as the open subset of $\overline{L}_{\vec{k}}$ characterized by the following five conditions:
\begin{itemize}
\item[(c1)]
the sheaf morphism $\alpha$ is a monomorphism;

\item[(c2)]
the sheaf morphism $\beta$ is an epimorphism; 

\item[(c3)]
the vector space morphisms $\alpha\otimes k(y)$ have maximal rank for all closed points $y\in\li$;

\item[(c4)]
if we consider the display associated with the monad $M(\alpha,\beta)$ as in eq.~\eqref{eq44}, and we restrict it to $\li$, after  twisting by $\Ol_{\li}(-1)$ and taking cohomology, we get a vector space morphism $\Phi:=H^0\left(\nu|_{\li}(-1)\right):H^0\left(\A|_{\li}(-1)\right)\longrightarrow H^0\left(\Wi(-1)\right)$; note that $h^0\left(\A|_{\li}(-1)\right)= h^0\left(\Wi(-1)\right)=nk_{3}$. We require that
\begin{equation*}
\det\Phi\neq0\,.
\end{equation*}
\item[(c5)] the cohomology $\E_{\alpha,\beta}$ of the monad $M(\alpha,\beta)$ is torsion-free.
\end{itemize}
We shall prove in Lemma \ref{6lm10}  that $\Lk$ is a smooth variety.  We let  $\E_{\alpha,\beta,\infty}=\left.\E_{\alpha,\beta}\right|_{\li}$.
\item
We consider the algebraic group
\begin{equation*}
\Gk=\Aut\left(\Uk\right)\times\Aut\left(\Vk\right)\times\Aut\left(\Wk\right).
\end{equation*}
This group acts naturally on $\Lk$ according to the following formulas:
\begin{equation}
\left\{
\begin{array}{rcl}
\alpha & \mapsto & \alpha'=\psi\alpha\phi^{-1}\\
\beta & \mapsto & \beta'=\chi\beta\psi^{-1}\\
\end{array}
\right.\qquad\bpsi=(\phi,\psi,\chi)\in\Gk\,.
\label{eqrho0}
\end{equation}
This action will be called $\rho_0:\Lk\times\Gk\to\Lk$.
\item
We introduce (see subsection \ref{pagPk}) a principal $\GL(r,\Com)$-bundle $\Pk\stackrel{\tau}{\longrightarrow}\Lk$ whose fibre over $(\alpha,\beta)$ is naturally identified with the space of framings at infinity for $\E_{\alpha,\beta}$, namely, a point $\eta\in\Pk$ is an isomorphism $\eta:\E_{\alpha,\beta,\infty}\stackrel{\sim}{\longrightarrow}\Ol_{\li}^{\oplus r}$, where $(\alpha,\beta)=\tau(\eta)$.
\item
One can lift $\rho_0$ to an action $\rho\colon\Pk\times\Gk\to\Pk$ by letting
\begin{equation}
\rho(\eta,\bar\psi) = \eta \circ \Lambda_\infty\left(\alpha,\beta;\bpsi\right)^{-1}\,,
\label{eqGonP}
\end{equation}
where $(\alpha,\beta)=\tau(\eta)$, and after letting $\left(\alpha',\beta'\right)=\bpsi\cdot(\alpha,\beta)$, the isomorphism 
\begin{equation}\label{lambda}\Lambda_\infty\left(\alpha,\beta;\bpsi\right):\E_{\alpha,\beta,\infty}\longrightarrow\E_{\alpha',\beta',\infty}
\end{equation} 
is  induced by $\bpsi:M(\alpha,\beta)\longrightarrow M\left(\alpha',\beta'\right)$.
The identity
$$\Lambda_\infty(\alpha,\beta,\bar\psi'\cdot\bar\psi) = \Lambda_\infty(\alpha',\beta',\bar\psi') \circ
\Lambda_\infty(\alpha,\beta,\bar\psi) $$
ensures that $\rho$ is indeed an action. We note that the projection $\tau\colon \Pk\to\Lk$ becomes a $\Gk$-equivariant morphism.
\end{itemize} \label{pagcond2}

We have now all ingredients needed to state our main result.
\begin{thm}[Main Theorem]
\label{teo2}
The set  $\M^n(r,a,c)$ is nonempty if and only if $$c+\frac{1}{2}na(a-1)\geq0\,.$$
 If this is the case, it can be  given the structure of a smooth algebraic variety of dimension $2rc+(r-1)na^2$ by representing it as the quotient 
 $\M^n(r,a,c)=\Pk/\Gk$. Moreover, $\M^n(r,a,c)$ turns out to be a fine moduli space of framed sheaves on $\Sigma_n$.
\end{thm}

Note that 
$$ \dim \M^n(r,a,c) = 2rc+(r-1)na^2 = 2r\Delta,$$
where $\Delta = c_2 - \frac{r-1}{2r}c_1^2$ is the discriminant of the sheaves parametrized by $ \M^n(r,a,c)$.

\begin{cor} The moduli space $\Mk$ is irreducible.
\end{cor}
\begin{proof} One knows by instanton counting that $\Mk$ is connected \cite{BPT}, hence, being smooth, it is irreducible as well.
\end{proof}

\medskip
\section{Families of framed sheaves}
In this section we   explain how the varieties $\Lk$ and $\Pk$ arise and construct a canonical family $\left(\tE,\tTheta\right)$ on the product $\Sigma_n\times\Pk$.

For any scheme $S$, let $T=\Sigma_n\times S$ and let $t_i$, $i=1,2$ be the projections onto the first and the second factor, respectively. Analogously, we introduce the product scheme $T_\infty=\li\times S$, together with the projections $u_i$, $i=1,2$.

\begin{defin}
Let $\vec{k}=(n,r,a,c)\in\Z^4$ with  $n\geq1$, $r\geq1$ and $0\leq a \leq r-1$. A coherent sheaf $\Fa$ on $T$    fulfills condition $\vec{k}$ if and only if it is flat on $S$ and for all  closed points $s\in S$
\begin{itemize}
 \item 
the restricted sheaf $\Fa_s$ is torsion-free and trivial at infinity on $T_s\simeq\Sigma_n$;
\item
the Chern character of $\Fa_s$ is $(r,aE,-c-\frac{1}{2}na^2)$.
\end{itemize}
\end{defin}
\begin{defin}
\label{defFam}
Given a vector $\vec{k}$, and a scheme $S$, an  $S$-family of framed sheaves on $\Sigma_n$  
is a pair $(\Fa,\Theta)$, where:
\begin{enumerate}
 \item 
$\Fa$ is a sheaf on $T=\Sigma_n\times S$  fullfilling condition $\vec{k}$;
\item
$\Theta$ is an isomorphism $\Fa|_{T_\infty}\to  \mathcal O_{T_\infty}^{\oplus r}$.
\end{enumerate}
Two $S$-families of framed sheaves on $\Sigma_n$ are isomorphic if they are isomorphic as framed sheaves on $T$ (cf.\ Definition \ref{DEf3}).
\end{defin}

For any sheaf $\Ga$ of $\Ol_T$-modules we let
\begin{equation*}
\Ga(p,q)=\Ga\otimes t_1^*\On(p,q)\qquad\text{for all $(p,q)\in\Z$}.
\end{equation*}
\begin{prop}
\label{teo1}
A sheaf  $\Fa$ on $T$ that  satisfies condition $\vec{k}$ is isomorphic to the cohomology of a  monad $M(\Fa)$ on $T$ 
\begin{equation}
M(\Fa):\qquad
\xymatrix{
0 \ar[r] & \Ua \ar[r]^A & \Va \ar[r]^B & \Wa \ar[r] & 0
}\,,
\label{eqMS}
\end{equation}
where the locally-free sheaves $\Ua$ and $\Wa$ are
\begin{equation*}
\begin{cases}
\begin{aligned}
\Ua&=\Ol_T(0,-1)\otimes t^{\ast}_2R^1t_{2*}\left[\Fa(-2,n-1)\right]\\
\Wa&=\Ol_T(1,0)\otimes t^{\ast}_2R^1t_{2*}\left[\Fa(-1,0)\right].
\end{aligned}
\end{cases}
\end{equation*}
The locally-free sheaf $\Va$ is defined as an extension
\begin{gather}
\xymatrix{
0 \ar[r] & \Va_- \ar[r]^{\ca} & \Va \ar[r]^{\da} & \Va_+ \ar[r] & 0\,,
}
\label{eqES}
\\
\text{where}\qquad
\begin{cases}
\begin{aligned}
\Va_+&:=t^{\ast}_2R^1t_{2*}\left[\Fa(-2,n)\right]\\
\Va_-&:=\Ol_T(1,-1)\otimes t^{\ast}_2R^1t_{2*}\left[\Fa(-1,-1)\right]
\end{aligned}
\end{cases}
\notag
\end{gather}
and the morphisms $\ca$ and $\da$ are  determined by $\Fa$.

\end{prop}
To prove this Proposition, one needs the following result.
\begin{lemma}\label{lm3}
Let $\mathcal{E}$ be a torsion-free sheaf on $\Sigma_n$, trivial at infinity. One has
\begin{align*}
 H^0(\mathcal{E}(p,q))=0\qquad&\text{\emph{for}}\qquad np+q\leq -1\,,
\\
 H^2(\mathcal{E}(p,q))=0\qquad&\text{\emph{for}}\qquad np+q\geq -(n+1)\,.
\end{align*}
\end{lemma}
\begin{proof} When $\mathcal{E}$ is locally free, the proof is essentially the same as in \cite[p.~24]{Ki}. Otherwise we get the thesis by using the injection $\E \rightarrowtail \E^{**}$.
\end{proof}
\begin{proof}[Proof of Proposition \ref{teo1}]
Let us consider the product scheme 
$\Sigma_n\times\Sigma_n\times S$, together with the canonical projections $p_{12}$, $p_{13}$ and $p_{23}$. Buchdahl \cite{Bu} proved the existence of a three-term locally-free resolution $\G^\bullet \twoheadrightarrow \Ol_\Delta$ 
of the structure sheaf $\Ol_{\Delta}$ of the diagonal $\Delta\subseteq \Sigma_n\times\Sigma_n$.  This is given by
$$ \left\{\begin{array}{rcl} \G^0 &=& \Ol_{\Sigma_n\times\Sigma_n} \\[3pt]
\G^{-1} &=& \mathcal R^\ast \\[3pt]
\G^{-2} & = & \Ol_{\Sigma_n}(-1,-1) \boxtimes \Ol_{\Sigma_n}(-1,n-1)
\end{array}\right.\ ;$$
here $\mathcal R$ is the extension
$$ 0 \to \Ol_{\Sigma_n}(1, 0) \boxtimes \Ol_{\Sigma_n}( 1,-n ) \to \mathcal R \to \Ol_{\Sigma_n}(0,1) \boxtimes \Ol_{\Sigma_n}(0,1) \to 0 $$
corresponding to the image in $H^1( \Sigma_n\times\Sigma_n,  \Ol_{\Sigma_n}(1,-1) \boxtimes \Ol_{\Sigma_n}(1,-n-1) )$ of $t$ under the connecting homomorphism of the exact sequence
\begin{multline*} 0 \to  \Ol_{\Sigma_n}(1,-1) \boxtimes \Ol_{\Sigma_n}(1,-n-1) ) \to \Ol_{\Sigma_n}(1,0) \boxtimes \Ol_{\Sigma_n}(1,-n) \\
\to \Ol_{\Sigma_n}(1,0) \boxtimes \Ol_{\Sigma_n}(1,-n) \otimes \Ol_Y\to 0\,. $$
\end{multline*} 
Here $Y$ is a suitable divisor in  $\Sigma_n\times\Sigma_n$, while $t$ is a   section of 
$ \Ol_{\Sigma_n}(0,1) \boxtimes \Ol_{\Sigma_n}(0,-1) \otimes \Ol_Y$ whose zero locus is $\Delta$.

Given any sheaf $\Fa$ on $T$ fullfilling condition $\vec{k}$, we introduce the complex
\begin{equation*}
\C^\bullet=\xymatrix{
\mathcal{C}^{-2} \ar[r] & \mathcal{C}^{-1} \ar[r] & \mathcal{C}^{0}
}=\left(p_{12}^*\G^\bullet\right)\otimes\left[p_{23}^{\ast}(\Fa(-1,0))\right]\,.
\end{equation*}
There are two spectral sequences, both abutting to the hyperdirect image $\mathbb{R}^\bullet p_{13\ast}(\mathcal{C}^{\bullet})$. From the first spectral sequence one gets
\begin{equation}
\mathbb{R}^ip_{13\ast}(\mathcal{C}^{\bullet})=
\begin{cases}
\Fa(-1,0)\qquad&\text{if}\;i=0\\
0\qquad&\text{otherwise}\,.
\end{cases}
\end{equation}
By using Lemma \ref{lm3}, one can obtain from the second exact sequence a complex that, when twisted by $\Ol_T(1,0)$, yields $M(\Fa)$.
\end{proof}
This proof implies that the sheaves $R^1t_{2*}\left(\Fa(p,q)\right)$ are locally free for $$(p,q)\in\Il=\{(-2,n-1),(-1,0),(-2,n),(-1,-1)\}\,.$$

\subsection{The variety $\Lk$}
As a straightforward consequence of Proposition \ref{teo1} we get the following result.

\begin{cor}
\label{corLocOnS}
Let the sheaf $\Fa$ and the monad $M(\Fa)$ be as in Proposition \ref{teo1}. Assume that $S$ is affine, and that the sheaves $R^1t_{2*}\left(\Fa(p,q)\right)$ are trivial for $(p,q)\in\Il$. There are isomorphisms
\begin{equation*}
\Ua\simeq t_1^*\Uk\,;\qquad
\Va\simeq t_1^*\Vk\,;\qquad
\Wa\simeq t_1^*\Wk\,,
\end{equation*}
the sheaves $\Uk$, $\Vk$ and $\Wk$ being defined as in eq.~\eqref{eqUVWk}.
\end{cor}
\begin{proof}
A trivialization for $R^1t_{2*}\left(\Fa(p,q)\right)$ amounts to choosing a closed point $s_0\in S$ and an isomorphism
\begin{equation*}
R^1t_{2*}\left(\Fa(p,q)\right)\stackrel{\sim}{\longrightarrow}\Ol_S\otimes\left[R^1t_{2*}\left(\Fa(p,q)\right)\otimes k(s_0)\right].
\end{equation*}
Since $\Fa_{s_0}$ is torsion-free and trivial at infinity, from Lemma \ref{lm3} and from the Semicontinuity Theorem one obtains the  isomorphism
\begin{equation*}
R^1t_{2*}\left(\Fa(p,q)\right)\otimes k(s_0)\simeq H^1\left(\Fa_{s_0}(p,q)\right)\,.
\end{equation*}
The dimensions of the vector spaces $H^1(\Fa_{s_0}(p,q))$ can be computed by means of Riemann-Roch Theorem and Lemma \ref{lm3}:
\begin{equation*}
h^1\left(\Fa_{s_0}(p,q)\right)=
\begin{cases}
k_1 & \text{for $(p,q)=(-2,n-1)$}\\
k_2 & \text{for $(p,q)=(-1,-1)$}\\
k_3 & \text{for $(p,q)=(-1,0)$}\\
k_4 & \text{for $(p,q)=(-2,n)$}\\
\end{cases}
\end{equation*}
where $k_i$, $i=1,\ldots,4$ are as in eq.~\eqref{eqki}. As a consequence the sheaves $R^1t_{2*}\left(\Fa(-2,n-1)\right)$, $R^1t_{2*}\left(\Fa(-1,-1)\right)$, $R^1t_{2*}\left(\Fa(-1,0)\right)$ and $R^1t_{2*}\left(\Fa(-2,n)\right)$ are free of ranks $k_1,\dots,k_4$ respectively, so that 
\begin{equation*}
\Ua
\simeq
\Ol_T(0,-1)^{\oplus k_1}\,;\qquad
\Wa
\simeq
\Ol_T(1,0)^{\oplus k_3}\,;\qquad
\Va_-
\simeq
\Ol_T^{\oplus k_4}\,;\qquad
\Va_+
\simeq
\Ol_T(1,-1)^{\oplus k_2}\,.
\end{equation*}
The thesis follows for $\Ua$ and $\Wa$. By plugging $\Va_-$ and $\Va_+$ into the sequence eq.~\eqref{eqES}, the latter splits, since
\begin{align*}
\Ext^1\left(\Ol_T(1,-1),\Ol_T\right)&\simeq H^1\left(T,\Ol_T(-1,1)\right)\simeq\\
\tag{$*$} &\simeq H^0\left(S,R^1t_{2*}\Ol_T(-1,1)\right)\simeq\\
&\simeq H^0\left(S,R^1t_{2*}\left[t_1^*\On(-1,1)\right]\right)\simeq\\
&\simeq H^0(S,\Ol_S)\otimes H^1\left(\Sigma_n,\On(-1,1)\right)=0
\end{align*}
(the second isomorphism holds true as $S$ is affine and the vanishing is a consequence of Lemma \ref{lm3}). This ends the proof.
\end{proof}

The following result is the absolute case of  Proposition \ref{teo1}, obtained by letting $S=\Spec\Com$, and follows easily from Corollary \ref{corLocOnS}.
\begin{cor}
\label{pro2}
Any sheaf $\mathcal{E}$ on $\Sigma_n $ that is torsion-free and trivial at infinity is isomorphic to the cohomology of a monad $\Mek(\E)$, which is of the form   $M(\alpha,\beta)$  for a suitable $(\alpha,\beta)\in\overline{L}_{\vec{k}}$. Note that we do not require $k_1\geq0$ \emph{a priori}.
As a consequence however, if $k_1=c+\frac{1}{2}na(a-1)<0$, the set $\M^n(r,a,c)$ is empty.
\end{cor}
 
We fix a $\vec{k}$ such that $k_1\geq0$.
One should note that $\Mek(\E)\simeq\Mek(\E')$ whenever $\E\simeq\E'$.
This provides a map between the set of isomorphism classes of torsion-free sheaves on $\Sigma_n$ that are  trivial at infinity  and the set of isomorphism classes of monads of the form $M(\alpha,\beta)$. The following two results establish the injectivity of this map, and enable us to characterize its image.
\begin{lemma}
\label{pro3}
Let $(\alpha,\beta), \left(\alpha',\beta'\right)$ be any two points in $\overline{L}_{\vec{k}}$ satisfying the conditions (c1) and (c2) introduced in section~\ref{Main result} (therefore $M=M(\alpha,\beta)$ and $M'=M\left(\alpha',\beta'\right)$ are monads). Then
\begin{equation*}
M\simeq M'\qquad\text{if and only if}\qquad \mathcal{E}(M)\simeq\mathcal{E}(M').
\end{equation*}
\end{lemma}
\begin{proof} The proof of \cite[Lemma 4.1.3]{Ok} holds true also when the cohomology sheaf of the monad is not locally free.
\end{proof}
\begin{prop}
\label{6prop3}
For any point $(\alpha,\beta)\in\overline{L}_{\vec{k}}$ satisfying conditions (c1) and (c2), the cohomology $\E$ of the monad $M(\alpha,\beta)$ is trivial at infinity if and only if the morphisms $(\alpha,\beta)$ satisfy conditions (c3) and (c4).
\end{prop}
\begin{proof}
Condition (c3) is equivalent to the local freeness of $\E|_{\li}$. As for condition (c4), the display of $M(\alpha,\beta)$ produces the exact sequence
\begin{equation*}
\xymatrix{
H^0\left(\E|_{\li}(-1)\right) \ar@{>->}[r] & H^0\left(\A|_{\li}(-1)\right) \ar[r]^-\Phi & H^0\left(\Wi(-1)\right) \ar@{->>}[r] & H^1\left(\E|_{\li}(-1)\right)\,.
}
\end{equation*}
Condition (c4) is equivalent to the vanishing of $H^i\left(\E|_{\li}(-1)\right)$, $i=0,1$. The thesis follows easily.
\end{proof}
This result enables us to identify $L_{\vec k}$ with the subset of $\Vek$ whose points correspond to cohomology sheaves $\E_{\alpha,\beta}$ that are torsion-free and trivial at infinity.
\begin{lemma}
\label{6lm10}
The variety $\Lk$ is smooth of dimension $\dim\Lk=\dim\Vek-\dim\Wek$, where $\Wek=\Hom\left(\Uk,\Wk\right)$.
\end{lemma}
\begin{proof}
We define the map
\begin{equation*}
\begin{array}{rccl}
\zeta:&\Vek&\longrightarrow&\Wek\\
&(\alpha,\beta)&\longmapsto&\beta\alpha.
\end{array}
\end{equation*}
So $\overline{L}_{\vec{k}}$ is the set $\{\zeta=0\}$. The differential $\de\zeta$ at the point $(\alpha_0,\beta_0)\in\Lk$ is the linear map
\begin{equation*}
\begin{array}{rccl}
(\de\zeta)|_{(\alpha_0,\beta_0)}:&\Vek&\longrightarrow&\Wek\\
&(\alpha,\beta)&\longmapsto&\beta_0\alpha+\beta\alpha_0.
\end{array}
\end{equation*}
The rank of this map is equal to $d=\dim\Vek-\dim\Lk$ on the (non empty) nonsingular locus of $\Lk$, and outside of this set is bounded above by $d$ (see \cite[pp.~31-33]{Har}).

One can prove as in \cite[Lemma II.4.1.7]{Ok} that
for any point $(\alpha,\beta)\in\Lk$, there is an isomorphism
\begin{equation*}
\coker(\de\zeta)|_{(\alpha,\beta)}\simeq H^2\left(\E_{\alpha,\beta}^*\otimes\E_{\alpha,\beta}\right).
\end{equation*}
Now, $\E_{\alpha,\beta}^*\otimes\E_{\alpha,\beta}$ is torsion-free and trivial at infinity, 
 so that $H^2\left(\E_{\alpha,\beta}^*\otimes\E_{\alpha,\beta}\right)=0$ by Lemma \ref{lm3} and $\de\zeta$ has maximal rank everywhere on $\Lk$. Since the singular locus of $\Lk$ coincides with the set of points in which $\de\zeta$ fails to have maximal rank, $\Lk$ is smooth.
\end{proof}

\subsection{The variety $\Pk$}
Let us introduce the varieties $\dT=\Sigma_n\times\Lk$ and $\dTi=\li\times\Lk$, together with the canonical projections shown in the following diagram:
\begin{equation*}
\xymatrix{
\li \ar@{^{(}->}[r] & \Sigma_n\\
\dTi  \ar@{^{(}->}[r] \ar[u]_-{\du_1} \ar[dr]_{\du_2} & \dT \ar[u]_{\dt_1} \ar[d]^{\dt_2}\\
 & \Lk\,.
}
\label{eqdT}
\end{equation*}
On $\dT$ we define the complex
\begin{equation*}
\xymatrix@C+1em{
\dMo:\qquad \dt_1^*\Uk \ar[r]^-{f_A} & \dt_1^*\Vk \ar[r]^-{f_B} & \dt_1^*\Wk 
}
\end{equation*}
where $f_A$ and $f_B$ are the defined by the formulas
\begin{equation*}
\begin{array}{rccl}
f_{A}\colon&(\dt_1^*\Uk)_{(x,\alpha,\beta)}&\longrightarrow&(\dt_1^*\Vk)_{(x,\alpha,\beta)}\\
&u&\longmapsto&\alpha(u)\\[7pt]
f_{B}\colon&(\dt_1^*\Vk)_{(x,\alpha,\beta)}&\longrightarrow&(\dt_1^*\Wk)_{(x,\alpha,\beta)}\\
&v&\longmapsto&\beta(v)
\end{array}\,;
\end{equation*}
here $u$ and $v$  are  points in the fibres over  $(x,\alpha,\beta)\in \Sigma_n\times\Lk$.
This complex is actually a monad by Lemma \ref{lmInjSurj}: let $\dE$ be its cohomology. This sheaf satisfies condition $\vec{k}$, and more precisely one has the following isomorphism for all points $(\alpha,\beta)\in\Lk$:
\begin{equation}\label{eqdEab}
\left(\dE\right)_{(\alpha,\beta)}\simeq\E_{\alpha,\beta}\,.
\end{equation}
The restriction $\dMoi$ of $\dMo$ to $\dTi$ is isomorphic to the monad
\begin{equation*}
\xymatrix@C+1em{
0 \ar[r] & \du_1^*\Ui \ar[r]^-{g_A} & \du_1^*\Vi \ar[r]^-{g_B} & \du_1^*\Wi \ar[r] & 0\,.
}
\end{equation*}
where $g_A$ and $g_B$ are the defined by the formulas
\begin{equation*}
\begin{array}{rccl}
g_{A}\colon&(\du_1^*\Ui)_{(x,\alpha,\beta)}&\longrightarrow&(\du_1^*\Vi)_{(x,\alpha,\beta)}\\
&u&\longmapsto&\alpha|_{\li}(u)\\[7pt]
g_{B}\colon&(\du_1^*\Vi)_{(x,\alpha,\beta)}&\longrightarrow&(\du_1^*\Wi)_{(x,\alpha,\beta)}\\
&v&\longmapsto&\beta|_{\li}(v)
\end{array}\,.
\end{equation*}
The cohomology of $\dMoi$  will be denoted by $\dEi$. It turns out that 
\begin{equation*} \left(\dEi\right)_{(\alpha,\beta)}\simeq \E_{\alpha,\beta,\infty}
\end{equation*}
for all points $(\alpha,\beta)\in\Lk$.

Let $\mathfrak N_{\vec k}$ denote the direct image $\breve{\mathfrak u}_{2\ast} \breve{\mathfrak E}_{\vec k,\infty}$. Since $\breve{\mathfrak E}_{\vec k,\infty}$
is a trivial vector  bundle on each fibre of $\breve{\mathfrak u}_{2}$, the sheaf $\mathfrak N_{\vec k}$  is locally free of rank $r$. Let $P_{\vec k}$
be its bundle of linear frames, \label{pagPk} which is a principal $\operatorname{GL}(r)$ bundle on $L_{\vec k}$. Moreover, if 
\begin{equation}\label{tau} 
\tau\colon P_{\vec k} \to L_{\vec k}
\end{equation}
 is the projection, the vector bundle $\widetilde{\mathfrak N}_{\vec k}=\tau^\ast \mathfrak N_{\vec k}$
is trivial. 
A point $\eta\in\Pk$ determines a framed sheaf $(\E,\theta)$
up to isomorphism, and two points in $\Pk$ corresponding to isomorphic framed sheaves are related by the action of an element of the group $\Gk$ of the automorphisms of the monad.
Indeed, in Section \ref{moduli} we shall construct the moduli space of framed sheaves as a quotient $\Pk/\Gk$.

\subsection{The family $\left(\tE,\tTheta\right)$}\label{sectEtT}
The geometrical environment of this subsection is provided 
by the varieties $\tT$ and $\tTi$, together with the canonical projections shown in the following diagram:
\begin{equation*}
\xymatrix{
\li \ar@{^{(}->}[r] & \Sigma_n\\
\tTi  \ar@{^{(}->}[r] \ar[u]_-{\ua_1} \ar[dr]_{\ua_2} & \tT \ar[u]_{\ta_1} \ar[d]^{\ta_2}\\
 & \Pk\,.
}
\end{equation*}
\begin{prop}
Let $\tE:=\left(\id_{\Sigma_n}\times\tau\right)^*\dE$.
This sheaf satisfies condition $\vec{k}$, and in particular for any point $\eta\in\Pk$ one has the natural isomorphism
\begin{equation}
\left(\tE\right)_\eta\simeq\E_{\tau(\eta)}\,.
\label{eqtEi}
\end{equation}
We shall call $\tEi$ the restriction of $\tE$ to $\tTi$, so that $\tEi\simeq\left(\id_{\Pu}\times\tau\right)^*\dEi\simeq\ua_2^*\tilde{\mathfrak N}_{\vec k}$.
\end{prop}
\begin{proof}
The flatness of $\tE$ on $\Pk$ follows from the identification $\tT\simeq\Pk\times_{\Lk}\dT$. The isomorphism \eqref{eqtEi} follows from eq. \eqref{eqdEab} since $(\tE)_\eta\simeq(\dE)_{\tau(\eta)}$. The last statement is trivial.
\end{proof}
One can extend the action $\rho$ of $\Gk$ on $\Pk$ to actions $\trho$ on $\tT$ and  $\rhoi$ on $\tTi$ 
by setting
\begin{equation*}
\left\{
\begin{aligned}
\trho&:=\id_{\Sigma_n}\times\rho\\
\rhoi&:=\id_{\li}\times\rho\,.
\end{aligned}
\right.
\end{equation*}
\begin{prop}
\label{proGontMo}
The sheaves $\tE$ and $\tEi$ are, respectively, isomorphic to the cohomologies of the monads
\begin{align*}
\tMo&:=\left(\id_{\Sigma_n}\times\tau\right)^*\dMo;\\
\tMoi&:=\left(\id_{\li}\times\tau\right)^*\dMoi.
\end{align*}
Both monads $\tMo$ and $\tMoi$ are $\Gk$-equivariant.
\end{prop}
\begin{cor}
The sheaf $\tE$ admits a $\Gk$-linearization $\Psi$  satisfying the isomorphism
\begin{equation*}
\left(\Psi|_{\tTi\times\Gk}\right)_{\left(\eta,\bpsi\right)}\simeq\Lambda_\infty\left(\alpha,\beta;\bpsi\right)
\end{equation*}
for any point $\left(\eta,\bpsi\right)\in\Pk\times\Gk$. Here $(\alpha,\beta)=\tau(\eta) $.
\end{cor}
\begin{proof} One defines a morphism of monads
$
m_{12}^*\dMo \to \left(\id_{\Sigma_n}\times\rho_0\right)^*\dMo
$
by letting
\begin{equation*}
\begin{array}{ccl}
(m_{12}^*\dMo)_{(x,\alpha,\beta,\bpsi)}&\longrightarrow&( \left(\id_{\Sigma_n}\times\rho_0\right)^*\dMo)_{(x,\alpha,\beta,\bpsi)}\\[5pt]
(u,v,w) & \longmapsto& (\phi(u),\psi(v),\chi(w))\end{array}\,.
\end{equation*}
where $m_{12}\colon \dT\times\Gk\to \dT$ is the canonical projection and $\rho_0$ was defined in eq.~\eqref{eqrho0}.
This induces by pullback a morphism 
$\Psi$ as in the statement of this Corollary.
\end{proof}

Since $\Pk$ is the bundle of linear frames of $\hN$, there exists a canonical isomorphism $\widetilde{\mathfrak{N}}_{\vec{k}}\longrightarrow\Ol_{\Pk}^{\oplus r}$, which can be regarded as a framing $\tTheta$ for the sheaf $\tE$. As a consequence of eq.~\eqref{eqGonP}, the morphism $\tTheta$ is $\Gk$-equivariant, namely, the following diagram   commutes
\begin{equation*}
\xymatrix@C+1em{
l_{12}^*\tEi \ar[r]^-{l_{12}^*\tTheta} \ar[d]_{\Psi_\infty} & \Ol^{\oplus r}_{\tTi\times\Gk}  \\
\rhoi^*\tEi \ar[ru]_-{\rhoi^*\tTheta}  
}
\label{eqtTh*}
\end{equation*}
where $\Psi_\infty=\Psi|_{\tTi\times\Gk}$ and $l_{12}:\tTi\times\Gk\longrightarrow\tTi$ is the projection.

\medskip
\section{The moduli space $\Mk$}\label{moduli}
In this section we   give the moduli space $\Mk$ a scheme structure, and  prove the first part of the Main Theorem \ref{teo2}. The space $\Mk$ can be set-theoretically identified with the quotient $\left.\Pk\right/\Gk$. We denote by $\pu:\Pk\longrightarrow\Mk$ the natural projection. 
\begin{thm}
\label{proMk}
The orbit space $\Mk$ is a smooth algebraic variety, and $\Pk$ is a locally trivial principal $\Gk$-bundle over it. 
\end{thm}
In order to prove this Theorem, we need to investigate some properties of the $\Gk$-action on $\Pk$.

\begin{lemma}\label{lm2}
Let $\mathcal{E}$ and $\mathcal{E}'$ be sheaves on $\Sigma_n$ that are torsion-free and trivial at infinity. There is an injection
\begin{equation*}
\xymatrix{0 \ar[r] &
\text{\emph{Hom}}(\mathcal{E},\mathcal{E}') \ar[r]^-{R} & \text{\emph{Hom}}(\mathcal{E}|_{\ell_{\infty}},\mathcal{E}'|_{\ell_{\infty}})\simeq\text{\emph{End}}(\mathbb{C}^r)
}
\label{eq380}
\end{equation*}
where $R$ is the restriction morphism.
\end{lemma}
\begin{proof} If $\mathcal{E}$ and $\mathcal{E}'$ are locally free, one has $\Hom\left(\E,\E'\right)\simeq H^0(\E^*\otimes\E')$. The sheaf $\E^*\otimes\E'$ is locally free and trivial at infinity, so that the result follows by twisting the structure sequence of $\li$ by it and  taking cohomology (see Lemma \ref{lm3}).
In   general, one can conclude because there is an injective morphism
$\Hom\left(\E,\E'\right)\to \Hom\left(\E^{\ast\ast},\E'^{\ast\ast}\right)$,
 and $\mathcal{E}^{\ast\ast}$ and $\mathcal{E}'^{\ast\ast}$ are locally free.
\end{proof}
This result generalizes to the relative situation. Let $S$  be a scheme, and let $\Fa$ and $\Fa'$ be two sheaves on $T=\Sigma_n\times S$ satisfying condition $\vec{k}$. 
\begin{cor}
\label{proInjInfty}
The restriction morphism
\begin{equation*}
\xymatrix{
\Hom\left(\Fa,\Fa'\right) \ar[r]^-R & \Hom\left(\Fa|_{T_\infty},\Fa'|_{T_\infty}\right)
}
\end{equation*}
is injective.
\end{cor}
\begin{proof}
Let $\Hl=\lHom\left(\Fa,\Fa'\right)$. Since both $\Fa$ and $\Fa'$ are locally free along $T_\infty$ one gets $\lTor_i\left(\Hl,\Ol_{T_\infty}\right)=0$ for $i>0$ (see for example \cite[p.~700]{GH}). Thus, if we twist the structure sequence of the divisor $T_\infty$ by $\Hl$, we get
\begin{equation*}
\xymatrix{
0 \ar[r] & \Hl(-T_\infty) \ar[r] & \Hl \ar[r] & \Hl|_{T_\infty} \ar[r] & 0\,,
}
\end{equation*}
where $\Hl|_{T_\infty}\simeq\lHom\left(\Fa|_{T_\infty},\Fa'|_{T_\infty}\right)$. It follows that
\begin{equation*}
\ker R=H^0\left(\Hl(-T_\infty)\right)=H^0\left(t_{2*}\left(\Hl(-T_\infty)\right)\right).
\end{equation*}
By Propositions \ref{teo1} and the first claim in Proposition \ref{teo1'}, the sheaf $\Hl(-T_\infty)$ is flat on $S$. Moreover, the second claim in Proposition \ref{teo1'} and Lemma   \ref{lm2} yield the following vanishing result for all closed points $s\in S$:
\begin{equation*}
\begin{split}
H^0\left(\Hl(-T_\infty)_s\right)&=H^0\left(\lHom\left(\Fa_s,\Fa'_s\right)(-\li)\right)=0\,.
\end{split}
\end{equation*}
The Semicontinuity Theorem entails the vanishing of the sheaf $t_{2*}\left(\Hl(-T_\infty)\right)$. This ends the proof.
\end{proof}

\begin{cor}
\label{cor1}
The action of $\Gk$ on $\Pk$ is free.
\end{cor}
\begin{proof}
Let $\left(\alpha,\beta;\bpsi\right)\in\Lk\times\Gk$, and put $\left(\alpha',\beta'\right)=\bpsi\cdot(\alpha,\beta)$. It follows from Lemma \ref{lm2} that a morphism $\Lambda\in\Hom\left(\E_{\alpha,\beta},\E_{\alpha',\beta'}\right) $ is fully determined by its restriction $\Lambda_\infty$ to $\ell_{\infty}$. It is not difficult to see \cite[Lemma 4.1.3]{Ok} that $\Lambda$ is induced by a unique isomorphism $\bpsi:M(\alpha,\beta)\longrightarrow M(\alpha',\beta')$ between the corresponding monads.

Whenever $\bpsi$ lies in the stabilizer of a point $\eta\in\Pk$, one has $\Lambda_\infty\left(\alpha,\beta;\bpsi\right)=\id_{\E_{\alpha,\beta,\infty}}$, where $(\alpha,\beta)=\tau(\eta)$. Since $\bpsi$ is uniquely determined, this implies $\bpsi=\id_{\Gk}$.
\end{proof}
\begin{prop}
\label{lmClos}
The graph $\Gamma$ of the action $\rho$ is closed in $\Pk\times\Pk$.
\end{prop}
\begin{proof}
Let $x=\left(\eta_{\alpha,\beta},\eta'_{\alpha',\beta'}\right)$ be  a point in $\Gamma$; by $\eta_{\alpha,\beta}$ we mean that $\eta$ belongs to the fibre over $(\alpha,\beta)$. One has
\begin{equation*}
\Lambda_\infty:=\left(\eta'_{\alpha',\beta'}\right)^{-1}\circ\eta_{\alpha,\beta}\in\Iso\left(\E_{\alpha,\beta,\infty},\E_{\alpha',\beta',\infty}\right)\,.
\end{equation*}
We define the vector space $\Hom_\Lambda\left(\E_{\alpha,\beta},\E_{\alpha',\beta'}\right)$ as the fibre product
\begin{equation}\label{square}
\xymatrix{
\Hom_\Lambda\left(\E_{\alpha,\beta},\E_{\alpha',\beta'}\right) \ar[r]^-i \ar[d]^j & \Com \ar[d]^-{\cdot\Lambda_\infty}\\
\Hom\left(\E_{\alpha,\beta},\E_{\alpha',\beta'}\right) \ar[r]^-{R} & \Hom\left(\E_{\alpha,\beta,\infty},\E_{\alpha',\beta',\infty}\right)\,,
}
\end{equation}
where $R$ is the restriction morphism to $\li$, and $\cdot\Lambda_\infty$ is the multiplication by $\Lambda_\infty$. Both morphisms $i$ and $j$ are injective, since $R$ is injective by Lemma \ref{lm2}, while $\cdot\Lambda_\infty$ is injective by the invertibility of $\Lambda_\infty$ \cite[Lemma 1.2]{Hi}.

Thus, $\Hom_\Lambda\left(\E_{\alpha,\beta},\E_{\alpha',\beta'}\right)$ is the subspace of homomorphisms between $\E_{\alpha,\beta}$ and $\E_{\alpha',\beta'}$ that at infinity reduce to multiples of $\Lambda_\infty$. By   \cite[Lemma 1.1]{Hi} one has the short exact sequence
\begin{multline}
0 \to 
\Hom_\Lambda\left(\E_{\alpha,\beta},\E_{\alpha',\beta'}\right) \to    \Com\oplus\Hom\left(\E_{\alpha,\beta},\E_{\alpha',\beta'}\right)  \\
\xrightarrow{(\cdot\Lambda_\infty,-(\cdot)|_{\li})}
\Hom\left(\E_{\alpha,\beta,\infty},\E_{\alpha',\beta',\infty}\right) \to 0 \,.
 \label{eqHomL}
\end{multline}
Since the morphism $i$ in the diagram \eqref{square} is injective, one has $\dim\Hom_\Lambda\left(\E_{\alpha,\beta},\E_{\alpha',\beta'}\right)\le 1$, 
so that $\Hom_\Lambda\left(\E_{\alpha,\beta},\E_{\alpha',\beta'}\right)$ is either zero or is generated by $\Lambda_\infty$. Hence,
\begin{equation}
\Gamma=\left\{\left.\left(\eta_{\alpha,\beta},\eta'_{\alpha',\beta'}\right)\in\Pk\times\Pk\,\right|\dim\Hom_\Lambda\left(\E_{\alpha,\beta},\E_{\alpha',\beta'}\right)=1\right\}\,.
\label{eqGamma}
\end{equation}
Let us consider the following product varieties, along with the associated canonical projections:
\begin{gather*}
\xymatrix{
\X:=\Sigma_n\times\Pk\times\Pk \ar@<1ex>[r]^-{q_{12}} \ar@<-1ex>[r]_-{q_{13}} \ar[rd]_-{q_{23}} & \tT\\
&\Pk\times\Pk  \,,
}\\
\xymatrix{
\Y:=\li\times\Pk\times\Pk \ar@<1ex>[r]^-{p_{12}} \ar@<-1ex>[r]_-{p_{13}} & \tTi\,.
}
\end{gather*}
One can pull-back the family $\left(\tE,\tTheta\right)$ to $\X$ in two different ways, getting $\left(q_{1i}^*\tE,p_{1i}^*\tTheta\right)$ with $i=2,3$. Out of these two pairs one defines
\begin{align*}
\Kl&=\lHom\left(q_{12}^*\tE,q_{13}^*\tE\right)\\
\Omega&=\left(p_{13}^*\tTheta\right)^{-1}\circ \left(p_{12}^*\tTheta\right)\in\Iso\left(p_{12}^*\tEi,p_{13}^*\tEi\right)\,.
\end{align*}
Since $q_{1i}^*\tE$ for $i=2,3$ are locally free along $\Y$, there is an isomorphism
\begin{equation*}
\Kl_\infty:=\lHom\left(p_{12}^*\tEi,p_{13}^*\tEi\right)\simeq\Kl|_{\Y}\,.
\end{equation*}
We introduce the sheaf $\Kl_{\Omega}$ by means of the exact sequence
\begin{equation}
\label{eqHl}
\xymatrix{
0 \ar[r] & \Kl_{\Omega} \ar[r] &
\Ol_\Y\oplus\Kl
\ar[rr]^-{
\left(
\cdot\Omega , -(\cdot)|_\Y
\right)
}
& &
\Kl_\infty \ar[r] & 0\,.
}
\end{equation}
By the second claim in Proposition \ref{teo1'}, for any point $x=\left(\eta_{\alpha,\beta},\eta'_{\alpha',\beta'}\right)\in\Pk\times\Pk$  the restriction of this sequence to the fibre of $q_{23}$ over $x$ is isomorphic to the sequence \eqref{eqHomL}.
In particular one gets  the isomorphism
\begin{equation}
H^0\left(\Kl_{\Omega,x}\right)\simeq\Hom_{\Lambda}\left(\E_{\alpha,\beta},\E_{\alpha',\beta'}\right).
\label{eqU}
\end{equation}
Since, by Proposition \ref{teo1'}, the sheaf $\Kl$ is flat on $\Pk\times \Pk$, the sheaf $\Kl_{\Omega}$ is flat on $\Pk\times\Pk$  as well \cite[Prop.~III.~9.1A.(e)]{Har}. Equation \eqref{eqU} and the Semicontinuity Theorem ensure that $\Gamma$, as characterized in \eqref{eqGamma}, is closed.
\end{proof}
The smooth algebraic varieties $\Pk$ and $\Gk$ have unique compatible structures of complex manifolds $\Pk^{an}$ and $\Gk^{an}$. Note that $\Gamma$ is closed in $\Pk^{an}\times\Pk^{an}$ as well.
\begin{cor}
The action of $\Gk^{an}$ on $\Pk^{an}$ is locally proper.
\end{cor}
\begin{proof}
Let $K_{\eta_0}$ be a compact neighbourhood  of a point $\eta_0\in\Pk^{an}$. We consider the morphism
\begin{equation*}
\begin{array}{rccl}
\gamma_{\eta_0}:&K_{\eta_0}\times\Gk^{an}&\longrightarrow&\Pk^{an}\times\Pk^{an}\\[5pt]
&\left(\eta;\bpsi\right)&\longmapsto&\left(\eta,\bpsi\cdot\eta\right)\,. \end{array}
\end{equation*}
Since the action of $\Gk^{an}$ is free, $\gamma_{\eta_0}$ is injective, so that its image is 
\begin{equation}
\im\gamma_{\eta_0}=\Gamma\cap\left(K_{\eta_0}\times\Pk^{an}\right)\,.
\label{eq20}
\end{equation}
We have to prove that the counterimage $\left(\rho|_{K_{\eta_0}\times\Gk^{an}}\right)^{-1}\!\!(K)$
of any compact subset $K\subset\Pk^{an}$ is compact. But it is easy to see that 
\begin{equation*}
\left(\rho|_{K_{\eta_0}\times\Gk^{an}}\right)^{-1}\!\!(K)=\gamma_{\eta_0}^{-1}\left(\Gamma\cap\left( K_{\eta_0}\times K\right)\right)\,.
\end{equation*}
As $\Gamma$ is closed by Proposition \ref{lmClos}, the thesis follows.
\end{proof}
We recall that an algebraic group $G$ is said to be \emph{special} if every locally isotrivial principal $G$-bundle is locally trivial \cite{Ser-EFA} (a fibration is said to be isotrivial if it is trivial in the \'etale topology).
\begin{lemma}
\label{lmspec}
The group $\Gk$ is special.
\end{lemma}
\begin{proof}
For any two positive integers $p,q$, let $H_{p,q}$ be the subgroup of $\GL(p+q,\Com)$ whose elements are the matrices
\begin{equation*}
\begin{pmatrix}
\bm{1}_q & A\\
0 & \bm{1}_p
\end{pmatrix}
\;,\quad\text{where}\ A\in\Hom\left(\Com^p,\Com^q\right)\,.
\end{equation*}
This group is isomorphic to the direct product of copies of the additive group $\Com$, and therefore it is special \cite[Prop.~1]{Grot-spec}. We have
\begin{equation*}
\Gk\simeq\GL(k_1,\Com)\times\Aut\left(\Vk\right)\times\GL(k_3,\Com)\,,
\end{equation*}
where $\Aut\left(\Vk\right)$ can be embedded as a closed subgroup in $\GL(nk_2+k_4,\Com)$. Moreover $H_{k_4,nk_2}$ is a normal subgroup of $\Aut\left(\Vk\right)$, and we get the short exact sequence of groups:
\begin{equation*}
\xymatrix{
1 \ar[r] & H_{k_4,nk_2} \ar[r] & \Gk \ar[r] & \GL(k_1,\Com)\times\GL(k_2,\Com)\times\GL(k_4,\Com)\times\GL(k_3,\Com) \ar[r] & 1\,.
} 
\end{equation*}
Since the group $\GL(p,\Com)$ is special for any $p$ \cite[Thm.~2]{Ser-EFA}, it turns out that $\Gk$ is special as well \cite[Lemma 6]{Ser-EFA}.
\end{proof}
We have now all ingredients to prove Theorem \ref{proMk}.
\begin{proof}[Proof of Theorem \ref{proMk}]
Since $\Mk$ is defined as a quotient set, the canonical projection $\pu$ induces both the quotient topology, which makes $\Mk$ into a noetherian topological space, and a canonical structure of   ringed space (see for example \cite{Se}).

Let $\Mk^{an}:=\Pk^{an}/\Gk^{an}$ and let $\pu^{an}:\Pk^{an}\longrightarrow\Mk^{an}$ be the projection. Since the action of $\Gk^{an}$ on $\Pk^{an}$ is free and locally proper, \cite[Satz 24]{Hol} implies that $\Mk^{an}$ with its natural structure of   ringed space is a complex manifold.

We have a commutative diagram of   ringed spaces:
\begin{equation*}
\xymatrix{
\Pk^{an} \ar[r] \ar[d]^{\pu^{an}} & \Pk \ar[d]^{\pu}\\
\Mk^{an} \ar[r] & \Mk\,.
}
\end{equation*}
It follows plainly that $\Mk$ is an algebraic variety, and is smooth since $\Mk^{an}$ is (see \cite[p.~109]{Tay}).

Moreover, the morphism $\Pk\times\Gk\longrightarrow\Pk\times\Pk$ defined by  $\rho$ is a closed immersion, and $\Mk$ is a geometric quotient of $\Pk$ modulo $\Gk$. It follows from \cite[Prop 0.9]{Mum} that $\Pk$ is a principal $\Gk$-bundle over $\Mk$, in particular it is locally isotrivial. Lemma \ref{lmspec} says that $\Pk$ is actually locally trivial, and this completes the proof.
\end{proof}

We can now prove the Main Theorem \ref{teo2}.
\begin{proof}[Proof of the Main Theorem, first part]
From Corollary \ref{pro2}, if $k_1<0$ the set $\Mk$ is empty. \textit{Vice versa}, let $\vec{k}=(n,r,a,c)$ be such that $k_1\geq0$ and $ 0\le a < r$, and define the sheaf $\E_{\vec{k}}$ as follows:
\begin{equation*}
\E_{\vec{k}}=
\left\{
\begin{aligned}
&\I_Z& \qquad &\text{if $r=1$} \\[5pt]
&\I_Z\oplus\On(E)^{\oplus a}\oplus\On^{\oplus(r-a-1)}&\qquad&\text{if $r>1$,}
\end{aligned}
\right.
\end{equation*}
where $\I_Z$ is the ideal sheaf of a 0-dimensional subscheme $Z\subset\Sigma_n$ of length $k_1$, whose support does not intersect 
$\ell_\infty$. 
The sheaf $\E_{\vec{k}}$ has a natural framing at infinity. Its
 Chern character   is $(r,aE,-c-\frac12na^2)$. It follows that $\Mk$ is empty if and only if $k_1<0$.

By Theorem \ref{proMk}, $\Mk$ is a smooth algebraic variety, and its dimension can be computed from the dimensions of $\Lk$, $\GL(r,\Com)$ and $\Gk$.
\end{proof}

\medskip
\section{The universal family}
In this section we show that  the moduli space $\Mk$ is fine by constructing a universal family of framed sheaves on $\Sigma_n$, i.e.,   we prove the second part of the Main Theorem 
\ref{teo2}.
Let us define the varieties $\Tu=\Sigma_n\times\Mk$ and $\Ti=\li\times\Mk$, together with the canonical projections shown in the following diagram:
\begin{equation*}
\xymatrix{
\li \ar@{^{(}->}[r] & \Sigma_n\\
\Ti  \ar@{^{(}->}[r] \ar[u]_-{\uu_1} \ar[dr]_{\uu_2} & \Tu \ar[u]_{\tu_1} \ar[d]^{\tu_2}\\
 & \Mk\,.
}
\end{equation*}
Let $\q \colon \tu_2^\ast\Pk \to \Tu$ and $\pr \colon \uu_2^\ast\Pk \to \Ti$ be the natural projections; note that 
$ \tu_2^\ast\Pk\simeq \tT$, $ \uu_2^\ast\Pk \simeq \tTi$, $\q=\id_{\Sigma_n}\times\pu$, $\pr=\id_{\li}\times\pu$,
where $\pu:\Pk\longrightarrow\Mk$ is the quotient morphism.
We define the sheaf
\begin{equation*}
\Eu=\left(\q_*\tE\right)^G
\end{equation*}
{on $\Tu$,} where $()^G$ denotes taking invariants with respect to the action of $\Gk$ on $\Pk$.
\begin{prop}
\label{proEu}
$\Eu$ is a rank $r$ coherent sheaf, satisfying condition $\vec{k}$. Actually, for any point $\left[\eta\right]\in\Mk$ with $\tau(\eta)=(\alpha,\beta)$ one has the isomorphism
\begin{equation}
\left(\Eu\right)_{\left[\eta\right]}\simeq\E_{\alpha,\beta}\,.
\label{eq10'}
\end{equation}
Furthermore, by considering the restriction at infinity $\Eui:=\left.\Eu\right|_{\Ti}$ we get
\begin{equation}
\Eui\simeq\left(\pr_*\tEi\right)^G\,.
\label{eq11'}
\end{equation}
\end{prop}
We need  a few preliminary results. First, we take the monad $\tMo=\left(\id_{\Sigma_n}\times\tau\right)^*\dMo$ as in Proposition \ref{proGontMo}, and we let
\begin{equation*}
\tMo=\quad
\xymatrix{
0 \ar[r] & \tU \ar[r]^-{f_{A}} & \tV \ar[r]^-{f_{B}} & \tW \ar[r] & 0
}
\end{equation*}
(the morphism $\tau$ was defined in Equation \eqref{tau}).
Analogously, we let
\begin{equation*}
\tMoi=\quad
\xymatrix{
0 \ar[r] & \tUi \ar[r]^-{g_{A}} & \tV \ar[r]^-{g_{B}} & \tWi \ar[r] & 0\,.
}
\end{equation*}
We introduce the subsheaves
\begin{equation}
\begin{aligned}
\Uu&=\left(\q_*\tU\right)^G\,;&
\Vu&=\left(\q_*\tV\right)^G\,;&
\Wu&=\left(\q_*\tW\right)^G\,;\\
\Uui&=\left(\pr_*\tUi\right)^G\,;&
\Vui&=\left(\pr_*\tVi\right)^G\,;&
\Wui&=\left(\pr_*\tWi\right)^G\,,
\end{aligned}
\label{eqGinv}
\end{equation}
\begin{lemma}
\label{lmLocUu}
The sheaves $\Uu$, $\Vu$ and $\Wu$ are locally free of rank $k_1$, $k_2+k_4$ and $k_3$, respectively. Furthermore, there are isomorphisms
\begin{align*}
\q^{*}\Uu&\simeq\tU\,;&
\q^{*}\Vu&\simeq\tV\,;&
\q^{*}\Wu&\simeq\tW\,;\\
\pr^{*}\Uui&\simeq\tUi\,;&
\pr^{*}\Vui&\simeq\tVi\,;&
\pr^{*}\Wui&\simeq\tWi.
\label{eqiso}
\end{align*}
\end{lemma}
\begin{proof}
We prove the thesis for the sheaf $\tU$, since the other cases are analogous.

We claim that there exists an open cover of $\tu_{2}^{*}\Pk$ over which $\tU$ can be $\Gk$-equivariantly trivialized.  Let $V\subseteq\M^{n}(r,a,c)$ be an open subset over which $\Pk$ trivializes. This implies the existence of a morphism
\begin{equation*}
q^{-1}(\Sigma_{n}\times V)\ni(x,\eta) \longmapsto (\phi(x,\eta),\psi(x,\eta),\chi(x,\eta))\in\Gk\,.
\end{equation*}
This enables us to define an automorphism $h$ of $\tU$ by the equation
\begin{equation*}
\begin{array}{rccl}
h \colon&(\tU)_{(x,\eta)}&\longrightarrow&(\tU)_{(x,\eta)}\\
& u &\longmapsto&\phi(x,\eta).u
\end{array}\,.
\end{equation*}
At the same time, let $g\colon\Uk|_{U}\longrightarrow\Ol_{U}^{\oplus k_{1}}$ be a trivialization of $\Uk$ over a suitable open subset $U\subseteq\Sigma_{n}$. On the intersection $W=q^{-1}(\Sigma_{n}\times V)\cap\ta^{-1}(U)=U\times\pi^{-1}(V)$ it makes sense to pursue the composition $(\ta_{1}^{*}g)\circ h $, and it is easy to prove that this isomorphism is a $\Gk$-equivariant trivialization for $\tU$ over $W$. Since we can cover 
{$\tu_{2}^{*}\Pk$} with such open subsets the claim follows.

The equivariant trivializations induce isomorphisms such $\Uu|_{U\times V}\longrightarrow\left(\q_*\Ol_{W}^{\oplus k_{1}}\right)^G=\Ol_{U\times V}^{\oplus k_{1}}$, and this proves the first statement.

To get the second statement, consider the natural morphism
\begin{equation}
\q^{*}\Uu\lhra \q^{*}\left(\q_{*}\tU\right)\longrightarrow\tU\,.
\label{eqNatMor}
\end{equation}
By using equivariant trivializations one can show that the restriction of this morphism to all subsets like $W$ is invertible, and this completes the proof.
\end{proof}
\begin{rem}
\label{LmUuinfty}
Since the restriction morphism $\cdot|_{\tTi}$ is $\Gk$-equivariant, the triples of sheaves $\left(\Uu,\Vu,\Wu\right)$ on $\Tu$ and $\left(\Uui,\Vui,\Wui\right)$ on $\Ti$ are related by restriction at infinity, that is, $\Uui\simeq\left.\Uu\right|_{\Ti}$, $\Vui\simeq\left.\Vu\right|_{\Ti}$ and $\Wui\simeq\left.\Wu\right|_{\Ti}$.
\end{rem}
We introduce now the ``universal'' monad.
\begin{prop}
\label{ProUnivMon}
One has the following commutative diagram of monads on $\Tu$:
\begin{equation}
\xymatrix{
\Mo: & 0 \ar[r] & \Uu \ar[r]^{\Au} \ar[d]^{\cdot|_{\Ti}} & \Vu \ar[r]^{\Bu} \ar[d]^{\cdot|_{\Ti}} & \Wu \ar[r] \ar[d]^{\cdot|_{\Ti}} & 0\\
\Moi: & 0 \ar[r] & \Uui \ar[r]^{\Aui} & \Vui \ar[r]^{\Bui} & \Wui \ar[r] & 0
}
\label{eqMon}
\end{equation}
where
\begin{equation*}
\left\{
\begin{aligned}
\Au&:=\left.\left(\q_*f_{A}\right)\right|_{\Uu}\\[5pt]
\Bu&:=\left.\left(\q_*f_{B}\right)\right|_{\Vu}
\end{aligned}
\right.\qquad\text{and}\qquad
\left\{
\begin{aligned}
\Aui&:=\left.\left(\pr_*g_{A}\right)\right|_{\Uui}\\[5pt]
\Bui&:=\left.\left(\pr_*g_{B}\right)\right|_{\Vui}\,.
\end{aligned}
\right.
\end{equation*}
The sheaves $\Eu$ and $\Eui$ are isomorphic to the cohomologies of the monads $\Mo$ and $\Moi$, respectively.
\end{prop}
\begin{proof}
The morphisms of $\Mo$ are well defined due to   the $\Gk$-equivariance of $f_{A}$ and $f_{B}$. The condition $\Bu\circ\Au=0$ follows from the functoriality of $\q_*$. The injectivity of $\Au$ is apparent since $\Au$ is the restriction of the injective morphism $\q_*f_{A}$.
 Lemma \ref{lmLocUu} implies $\q^*\Bu\simeq f_{B}$ and $\q^*\left(\coker\Bu\right)\simeq\coker f_{B}=0$; the vanishing of $\coker\Bu$ follows from the faithful flatness of $\q$.

The proof for $\Moi$ is analogous, and the commutativity of the diagram is an easy consequence of  Remark \ref{LmUuinfty}.

We prove now that $\Eu$  is the cohomology $\Mo$. The analogous result for $\Eui$ is proved in the same way. Let $\mathscr{E}$ be the cohomology of $\Mo$.  By applying Proposition \ref{proGontMo} to the display of $\Mo$ one obtains a natural injection $\iota\colon\mathscr{E}\lhra\Eu$.
We claim that $\iota$ is invertible. By pulling-back the display of $\Mo$ to {$\tu_{2}^{*}\Pk$} and   applying Lemma \ref{lmLocUu} one gets an isomorphism $\kappa\colon\q^{*}\mathscr{E}\longrightarrow\tE$ which fits into a commutative triangle
\begin{equation}
\xymatrix{
\q^{*}\mathscr{E} \ar[r]^-{\q^{*}\iota} \ar[dr]_{\kappa} & \q^{*}\Eu \ar[d]^{\varphi}\\
& \tE
}
\label{eqTrIota}
\end{equation}
where $\varphi$ is defined analogously to \eqref{eqNatMor}. 
Actually $\varphi$ is an isomorphism, as follows directly from the definition of $\Eu = (\q_\ast\tE)^{\Gk}$. Indeed, 
since the matter is local on $\Xi= \Sigma_n\times\Mk$, we can replace the latter with an open affine subset and apply Lemma A.6 in \cite{stackypeople} (note that $\q$ is faithfully flat).
We deduce from the diagram \eqref{eqTrIota} that $\q^{*}\iota$ is an isomorphism, and since $\q$ is faithfully flat,  $\iota$ is invertible as well. This completes the proof.
\end{proof}
\begin{proof}[Proof of Proposition \ref{proEu}]
The coherence of $\Eu$ follows from $\Eu\simeq\E\left(\Mo\right)$, and its rank can be computed easily from the ranks of $\Uu$, $\Vu$ and $\Wu$ given by Lemma \ref{lmLocUu}. 
From Lemma \ref{lmLocUu} and Proposition \ref{ProUnivMon} we get isomorphisms $\q^*\Mo\simeq\tMo$, $\q^*\Eu\simeq\tE$ and $\q^*\Au\simeq\tA$.
Note that $\left(\Au\right)_{\left[\eta\right]}=\alpha$ for all points $\left[\eta\right]\in\Mk$, with $\tau(\eta)=(\alpha,\beta)$. The flateness of $\Eu$ follows from \cite[Lemma 2.1.4]{HL-book} applied
to the analogue of the sequence \eqref{lemma2.2} for the monad $\Mo$. This is enough to show that $\Eu$ satisfies condition $\vec{k}$.

Finally, eq.~\eqref{eq11'} is a consequence of the commutativity of the diagram  \eqref{eqMon}.
\end{proof}
The morphism $\tTheta$ (cf.~subsection \ref{sectEtT}) provides a framing for the sheaf $\Eu$. Note that this morphism is $\Gk$-equivariant.\begin{defin}
We define the isomorphism $\Thu$ as the restriction of $\pr_*\tTheta$ to the $\Gk$-invariant subsheaf $\Eui$:
\begin{equation*}
\Thu:=\left.\left(\pr_*\tTheta\right)\right|_{\Eui}:\Eui\stackrel{\sim}{\longrightarrow}\Ol_{\Ti}^{\oplus r}\,.
\end{equation*}
We shall call $\Thu$ the universal framing.
\end{defin}

We show how to  associate a scheme morphism $f_{\left[(\Fa,\Theta)\right]}:S\longrightarrow\Mk$   with an isomorphism class $\left[(\Fa,\Theta)\right]$  of families of framed sheaves on $T=\Sigma_n\times S$.
We begin by describing some properties of the monad $M(\Fa)$.
\begin{lemma}
\label{lmAsBs}
Let $\Fa$ be a sheaf on $T$ which satisfies condition $\vec{k}$, and let $M(\Fa)$ be  a monad for it.
\begin{itemize}
 \item
For any closed point $s\in S$ there is an isomorphism of complexes
\begin{align}
M(\Fa)_s&\simeq M(\alpha(s),\beta(s))\,,
\label{eqMF}
\end{align}
where $(\alpha(s),\beta(s))=\left(A_s,B_s\right)\in\Lk$.
\item
The restriction $M(\Fa)_\infty$ of the monad $M(\Fa)$ to $T_\infty$ is  a monad, whose cohomology is isomorphic to $\Fa|_{T_\infty}$. For any closed point $s\in S$ there is an isomorphism
\begin{equation}
M(\Fa)_{\infty,s}\simeq\quad\xymatrix@C+1em{
  0 \ar[r] & \Ui \ar[r]^-{\alpha(s)|_{\li}} & \Vi \ar[r]^{\beta(s)|_{\li}} & \Wi \ar[r] & 0\,.
}
\label{eqMFi}
\end{equation}
\end{itemize}
\end{lemma}
\begin{proof}
The proof splits in two steps.
1. 
Let $s\in S$ be any closed point. By Corollary \ref{corLocOnS}, there is an isomorphism
\begin{equation*}
M(\Fa)_s\simeq
\xymatrix{
\Uk \ar[r]^{A_s} & \Vk \ar[r]^{B_s} & \Wk\,.
}
\end{equation*}
It is enough to show that $A_s$ is injective. This is a consequence of the application of \cite[Lemma 2.1.4]{HL-book} to the short exact sequence.
\begin{equation*}
\xymatrix{
0 \ar[r] & \Ua \ar[r]^-A & \ker B \ar[r] & \Fa \ar[r] & 0\,.
}
\end{equation*}
Hence $M(\Fa)_s$ is a monad, whose cohomology is isomorphic to $\Fa_s$; the latter sheaf is torsion free and   trivial at infinity. Proposition \ref{6prop3} implies $\left(A_s,B_s\right)\in\Lk$.

2. Equation \eqref{eqMFi} follows from eq.~\eqref{eqMF}, since the condition $(\alpha(s),\beta(s))\in\Lk$ implies that $\alpha(s)|_{\li}$ is injective. By Lemma \ref{lmInjSurj} this condition ensures that $A_\infty$ is injective.
\end{proof}
\begin{rem}
Suppose that $(\Fa,\Theta)$  is a family of framed sheaves on $T$. For any closed point $s\in S$, let $\theta(s)=\Theta_s$ be the restricted framing. One has
\begin{equation*}
\theta(s)\in\Iso\left(\E_{\alpha(s),\beta(s),\infty},\Ol_{\Pu}^{\oplus r}\right)\,.
\end{equation*}
\end{rem}
We proceed with the construction of the morphism $f_{\left[(\Fa,\Theta)\right]}:S\longrightarrow\Mk$ by defining it first on closed points. We choose an open affine cover $S=\bigcup_{a\in\Al}S_a=\bigcup_{a\in\Al}\Spec\Sl_a$ where each $S_a$'s satisfies the conditions required of the scheme $S$  in  Corollary \ref{corLocOnS}.
Moreover, if $(A,B)$ is the pair of morphisms in the monad $M(\Fa)$, we introduce the following notation:
\begin{equation*}
\left(A_a,B_a,\Theta_a\right):=\left(A|_{\Sigma_n\times S_a},B|_{\Sigma_n\times S_a},\Theta|_{\li\times S_a}\right)\,.
\end{equation*}
Recall that $t_2:T\longrightarrow S$ is the projection. Note that by applying the functor $t_{2*}$ to the monad $M(\Fa)$ restricted to $\Sigma_n\times S_a$ we obtain a complex of trivial sheaves on $S_a$, so that
\begin{equation*}
(t_{2*}A_a,t_{2*}B_a)\in\Vek\otimes\Sl_a\simeq\Sl_a^{\oplus d}\,.
\end{equation*}
If we define $(\alpha_a(s),\beta_a(s))=\left(t_{2*}A_a,t_{2*}B_a\right)\otimes_{\Sl_a}k(s)$ we obtain the same morphisms as in Lemma \ref{lmAsBs}. We define a morphism $\fb_a:S_a\longrightarrow\Lk$ by letting $\fb_a(s)=(\alpha_a(s),\beta_a(s))$. We complete this to a scheme morphism by defining the ring homomorphism
\begin{equation*}
\fb_a^\sharp:\Com[z_1,\dots,z_d]\longrightarrow\Sl_a\\
\end{equation*}
which maps the polynomial $g$ to $\tilde{g}\left(A_a,B_a\right)$, where $\tilde{g}$ is the natural extension of $g$ to the ring $\Sl_a[z_1,\dots,z_d]$.

In the same way, we define $\theta_a(s)=u_{2*}\Theta\otimes_{\Sl_a}k(s)$. This allows us to lift the morphisms $\fb_a$ to morphisms $\tf_a:S_a\longrightarrow\Pk$ which we define on closed points as $\fb_a(s)=\theta_a(s)$. Again, these extend to scheme morphisms. By composing these morphisms with the projection $\pu:\Pk\longrightarrow\Mk$ we obtain morphisms $f_a:S_a\longrightarrow\Mk$, which glue to a morphism $f_{\left[(\Fa,\Theta)\right]}:S\longrightarrow\Mk$, since on overlaps the different $\tf_a$'s differ by the action of $\Gk$.

\subsection{Fineness}
In this section we prove that $\Mk$ represents a moduli functor, i.e., it is a fine moduli space. Let $\Sc$ be the category of noetherian reduced schemes of finite type over $\Com$ and $\Sets$ the category of sets.
\begin{defin}
For any vector $\vec{k}$ such that $k_1\geq0$ we introduce the contravariant functor $\Mod:\Sc\longrightarrow\Sets$ by the following prescriptions:
\begin{itemize}
 \item
for any object $S\in\Ob\left(\Sc\right)$ we define the set $\Mod(S)$ as
\begin{equation*}
\Mod(S):=\frac
{\{ \text{$S$-families $(\Fa,\Theta)$ of framed sheaves on $\Sigma_n$ } \}}
{\{\text{isomorphisms}\}}\,; 
\end{equation*}
\item
for any morphism $\varphi:S'\longrightarrow S$ we define the set-theoretic map
\begin{equation*}
\begin{array}{rccl}
\Mod(\varphi):&\Mod(S)&\longrightarrow&\Mod(S')\\
&\left[(\Fa,\Theta)\right]&\longmapsto&\left[(\varphi_\Sigma^*\Fa,\varphi_\infty^*\Theta)\right]\,, 
\end{array}
\end{equation*}
where $\varphi_\Sigma=\id_{\Sigma_n}\times\varphi$ and $\varphi_\infty=\id_{\li}\times\varphi$.
\end{itemize}
\end{defin}
(Notations such as $\varphi_\Sigma$ and $\varphi_\infty$ will be used repeatedly in the following.)
Observe that $\Mod\left(\Spec\Com\right)$ is the set underlying $\Mk$. The key property of this functor is the following.
\begin{prop}
\label{proFinMk}
The functor $\Mod(-)$ is represented by the scheme $\Mk$, that is, there is a natural isomorphism of functors
\begin{equation*}
\Mod(-)\simeq\Hom\left(-,\Mk\right)\,.
\end{equation*}
This implies that $\Mk$ is a fine moduli space of framed sheaves on $\Sigma_n$. The pair $\left(\Eu,\Thu\right)$ on $\Sigma_n\times\Mk$ is a universal family of framed sheaves on $\Sigma_n$.
\end{prop}
We divide the proof of this Proposition in three Lemmas.

\begin{lemma}
Let $S$ be any scheme.
\begin{itemize}
 \item 
We define the map $\Ma_S$ as
\begin{equation*}
\begin{array}{rccl}
\Ma_S:&\Mod(S)&\longrightarrow&\Hom\left(S,\Mk\right)\\
&\left[(\Fa,\Theta)\right]&\longmapsto&f_{\left[(\Fa,\Theta)\right]}\,.
\end{array} 
\end{equation*}
\item
We define the map $\oMa_S$ as
\begin{equation*}
\begin{array}{rccl}
\oMa_S:&\Hom\left(S,\Mk\right)&\longrightarrow&\Mod(S)\\
&f&\longmapsto&\left[\left(f_\Sigma^*\Eu,f_\infty^*\Thu\right)\right]\,.
\end{array}
\end{equation*}
\end{itemize}
In this way we get natural tranformations $\Ma:\Mod(-)\longrightarrow\Hom\left(-,\Mk\right)$ and $\oMa:\Hom\left(-,\Mk\right)\longrightarrow\Mod(-)$.
\end{lemma}
\begin{proof}
The naturality of $\oMa$ is straightforward since
\begin{equation*}
\left[\left(\varphi_\Sigma^*f_\Sigma^*\Eu,\varphi_\infty^*f_\infty^*\Thu\right)\right]=\left[\left(\left(f\circ\varphi\right)_\Sigma^*\Eu,\left(f\circ\varphi\right)_\infty^*\Thu\right)\right]\,,
\end{equation*}
whenever a composition of morphisms $\xymatrix{S' \ar[r]^-\varphi & S \ar[r]^-f &\Mk}$ is given.

To prove the naturality of $\Ma$ we need to show that, given any morphism $\xymatrix{S' \ar[r]^\varphi & S}$ and any family $\left(\Fa,\Theta\right)$  on $T$, the following equality holds true:
\begin{equation}
f_{\left[\left(\Fa,\Theta\right)\right]}\circ \varphi=f_{\left[\left(\varphi_\Sigma^*\Fa,\varphi_\infty^*\Theta\right)\right]}\,. \label{eqMa}
\end{equation}
To simplify the notation we let $f=f_{\left[\left(\Fa,\Theta\right)\right]}$ and $f'=f_{\left[\left(\varphi_\Sigma^*\Fa,\varphi_\infty^*\Theta\right)\right]}$. We can assume  $S=\Spec\Sl$ and $S'=\Spec\Sl'$, so that $\varphi$ is induced by a ring homomorphism $\xymatrix{\Sl \ar[r]^-{\varphi^\sharp} & \Sl'}$.
If $(A,B)$ and $\left(A',B'\right)$ are the morphisms in the monads $M(\Fa)$ and $M\left(\varphi^*\Fa\right)$, respectively, one has
\begin{equation*}
\left(A',B'\right)=\left(\varphi^\sharp\right)^{\oplus d}(A,B)
\end{equation*}
and $\fb\circ\varphi=\fb':S'\longrightarrow\Lk$.
We can assume that $\im\fb'$ is contained in  an open affine subset $V=\Spec\Ll$ of $\Lk$ which  trivializes $\hN$, and we put $N_{\vec{k}}=\Gamma\left(V,\hN\right)$ which turns out to be a free $\Ll$-module of rank $r$.

We put $N=\Gamma\left(T_\infty,\Fa|_{T_\infty}\right)$ which is a free $\Sl$-module of rank $r$, and similarly we put $N'=\Gamma\left(T'_\infty,\varphi_\infty^*\left(\Fa|_{T_\infty}\right)\right)$ which is a free $\Sl'$-module of the same rank; By the homomorphisms $f^\sharp$ and $ \varphi^\sharp$ one has $N\simeq N_{\vec{k}}\otimes_\Ll\Sl$ and $N'\simeq N\otimes_\Sl\Sl'$. Moreover one has $\varphi_\infty^*\Theta\simeq\Theta\otimes_\Sl1_{\Sl'}$.

For any polynomial $h\in\Sym_{\Ll}\left(H^*\right)$, where $H\simeq\Hom_\Ll\left(N_{\vec{k}},\Ll^{\oplus r}\right)$, we get
\begin{equation*}
\begin{split}
\ev_{\varphi_\infty^*\Theta}\left(h\otimes_\Ll1_{\Sl'}\right)&=\ev_{\left(\Theta\otimes_\Sl1_{\Sl'}\right)}\left[\left(h\otimes_\Ll1_{\Sl}\right)\otimes_\Sl1_{\Sl'}\right]=\\
&=\left[\ev_\Theta\left(h\otimes_\Ll1_{\Sl}\right)\right]\otimes_\Sl1_{\Sl'}
\end{split}
\end{equation*}
where $\ev$ is the evaluation of polynomials.
Hence
\begin{equation*}
\tf\circ\varphi=\tf':S'\longrightarrow\Pk\,.
\end{equation*}
By applying the projection $\pu$ to both sides of this equation, we obtain eq.~\eqref{eqMa}.
\end{proof}

\begin{lemma}
\label{lmMaoMa}
For any scheme $S$ one has $\oMa_S\circ\Ma_S=\id_{\Mod(S)}$.
\end{lemma}
\begin{proof}
We need to prove that
\begin{equation*}
\left(f_\Sigma^*\Eu,f_\infty^*\Thu\right)\simeq\left(\Fa,\Theta\right)\,,
\end{equation*}
for any $S$-family $\left(\Fa,\Theta\right)$  of framed sheaves on $\Sigma_n$, where $f=f_{\left[\left(\Fa,\Theta\right)\right]}$. It is enough to show that there is an isomorphism
\begin{equation}
f_\Sigma^*\Mo\simeq M(\Fa)
\label{eqeq}
\end{equation}
and that this isomorphism is compatible with the framings. By Lemma \ref{lmLocUu} there are isomorphisms $f_\Sigma^*\Uu\simeq t_1^*\Uk$, $f_\Sigma^*\Vu\simeq t_1^*\Vk$, $f_\Sigma^*\Wu\simeq t_1^*\Wk$, together  with
 $\left(\q^*\Au,\q^*\Bu,\pr^*\Thu\right)$ $\simeq\left(\tA,\tB,\tTheta\right)$.
When $S$ is affine and satisfies the conditions of Corollary \ref{corLocOnS}, we have in addition the isomorphisms $\Ua\simeq t_1^*\Uk$, $\Va\simeq t_1^*\Vk$ and $\Wa\simeq t_1^*\Wk$, and we have
\begin{equation*}
\begin{split}
f_\Sigma^*\left(\Au,\Bu\right)&=\tilde{f}_\Sigma^*\q^*\left(\Au,\Bu\right)=\tilde{f}_\Sigma^*\left(\tA,\tB\right)=\tilde{f}_\Sigma^*\ta_2^*\tau^*\left(\id_{\Lk}\right)=t_2^*\tf^*\tau^*\left(\id_{\Lk}\right)\\
&=t_2^*\fb^*\left(\id_{\Lk}\right)=(A,B)\,.
\end{split}
\end{equation*}
where $\tilde{f}_\Sigma=\id_{\Sigma_n}\times\tilde{f}$.
This proves eq.~\eqref{eqeq} locally, and similarly the compatibility of $\Theta_{\vec{k}}$ can be shown. By Corollary \ref{proInjInfty} we get the thesis.
\end{proof}

\begin{lemma}
\label{lm6}
For any vector $\vec{k}$ such that $k_1\geq0$ and for any scheme $S$ one has
\begin{equation*}
\Ma_S\circ\oMa_S=\id_{\Hom\left(S,\Mk\right)}\,.
\end{equation*}
\end{lemma}
\begin{proof}
Let $g:S\longrightarrow\Mk$ be any scheme morphism.
 We need to show that
\begin{equation*}
g=f_{\left[\left(g_\Sigma^*\Eu,g_\infty^*\Thu\right)\right]}\,;
\end{equation*}
for simplicity we set $f=f_{\left[\left(g_\Sigma^*\Eu,g_\infty^*\Thu\right)\right]}$. Let $M\left(g_\Sigma^*\Eu\right)$ be a monad
\begin{equation*}
\xymatrix{
0 \ar[r] & \Ua \ar[r]^A & \Va \ar[r]^B & \Wa \ar[r] & 0
}
\end{equation*}
  associated with $g_\Sigma^*\Eu$. We can work locally by assuming that $S=\Spec\Sl$ satisfies the hypotheses of Corollary \ref{corLocOnS} for the sheaves $\Ua$, $\Va$, $\Wa$ and that $\im g\subseteq W$, where $W\subseteq \Mk$ is a trivializing open subset for the $\Gk$-principal bundle $\Pk$. Thus, there exists a local section $\sigma:W\longrightarrow\Pk$ lifting $g$ to $\Pk$:
\begin{equation}
\xymatrix{
& \Pk \ar[d]^\pu\\
S \ar[ru]^{\sigma\circ g} \ar[r]_-g & \Mk\,.
}
\label{eqLiftg}
\end{equation}
Under our assumptions, the complex
\begin{equation*}
\xymatrix@C+1em{
 g_\Sigma^*\Mo:&  0 \ar[r] &  g_\Sigma^*\Uu \ar[r]^{g_\Sigma^*\Au} & g_\Sigma^*\Vu \ar[r]^{g_\Sigma^*\Bu} & g_\Sigma^*\Wu \ar[r] & 0
}
\end{equation*}
is a monad. Indeed the morphism $g_\Sigma^*\Au$ is injective, as it follows from diagram \eqref{eqLiftg} and Lemma \ref{lmInjSurj}. This monad is isomorphic to $M\left(g_\Sigma^*\Eu\right)$: as a matter of fact, their cohomologies are isomorphic and \cite[Lemma 4.1.3]{Ok} applies. Because of this isomorphism, in view of Proposition \ref{proEu} one has
\begin{equation*}
\left(g_\Sigma^*\Eu\right)_s\simeq\E_{\alpha(s),\beta(s)},
\end{equation*}
where $g(s)=[\theta(s)]$ and  $(\alpha(s),\beta(s))=\tau(\theta(s))$ for any closed point $s\in S$. This ends the proof.
\end{proof}


%

\end{document}